\documentclass{amsart}

\usepackage{amsmath}
\usepackage{amsfonts}
\usepackage{amssymb}
\usepackage{amscd}
\usepackage{amsthm}
\usepackage{setspace}
\usepackage{graphicx}
\usepackage{hyperref}

\newtheorem*{mainthm}{Main Theorem}

\everymath{\displaystyle}

\newtheorem{theorem}{Theorem}[section]
\newtheorem{lemma}[theorem]{Lemma}

\newtheorem{proposition}[theorem]{Proposition}

\newtheorem{claim}[theorem]{Claim}

\theoremstyle{definition}
\newtheorem{remark}[theorem]{Remark}
\newtheorem{example}[theorem]{Example}
\newtheorem{definition}[theorem]{Definition}

\newtheorem{notation}[theorem]{Notation}

\newcommand{\Frob}{\textrm{Frob}}
\newcommand{\Gal}{\textrm{Gal}}

\newcommand{\Q}{\mathbb{Q}}
\newcommand{\R}{\mathbb{R}}
\newcommand{\F}{\mathbb{F}}

\newcommand{\Z}{\mathbb{Z}}
\newcommand{\Hom}{\operatorname{Hom}}
\newcommand{\Ad}{\textrm{ad}}
\newcommand{\e}{{\bf\textrm{e}}}
\newcommand{\B}{\mathfrak{B}}
\newcommand{\m}{\mathfrak{m}}
\newcommand{\rep}{\textrm{Rep}}

\newcommand{\T}{\pmb{t}}
\newcommand{\rb}{\overline\rho}
\newcommand{\gen}{\textrm{gen}}

\newcommand{\cD}{\mathcal{D}}

\newcommand{\cF}{\mathcal{F}}
\numberwithin{equation}{section}

\begin{document}

\title{Lifting $N$-dimensional  Galois representations to characteristic zero}
\author{Jayanta Manoharmayum}
\address{School of Mathematics and Statistics, University of
    Sheffield,
    Sheffield S3 7RH,
    U.K.}
\email{J.Manoharmayum@sheffield.ac.uk}
\date{}
\subjclass[2000]{11F80}

\begin{abstract} Let $F$ be a number field, let $N\geq 3$ be an integer, and let $k$ be a
finite field of characteristic $\ell$. We show that if
 $\rb:G_F\longrightarrow GL_N(k)$ is a continuous representation
with image of $\rb$ containing $SL_N(k)$ then, under moderate conditions at primes dividing
$\ell\infty$,  there is a continuous representation
$\rho:G_F\longrightarrow GL_N(W(k))$ unramified outside finitely many primes with
$\rb\sim\rho\ \textrm{mod}\ \ell$.\end{abstract}

\maketitle

\section{Introduction}
Let $F$ be a number field. Suppose we are given a continuous $\ell$-adic
representation
$\rho:G_F\longrightarrow GL_N(\overline\Q_\ell)$
unramified outside finitely many places. The representation $\rho$ then takes values in
a finite
extension of $\Q_\ell$, and on reducing modulo $\ell$  a stable lattice one gets a
continuous representation $\rb: G_F\longrightarrow GL_N(\overline\F_\ell)$,
 unique up to semi-simplification. Conversely,
one can ask if a given mod $\ell$ representation $\rb: G_F\longrightarrow GL_N(\overline\F_\ell)$ is necessarily the reduction of
an $\ell$-adic representation. This was answered in the affirmative  when $N=2$ by  Ramakrishna
in
\cite{ramakrishna}. In this article, we generalise the
method of
Ramakrishna, {\it loc. cit.}, to $N\geq 3$ and provide  an answer to the finding
characteristic
zero lifts when the image of $\rb$  and the residue characteristic $\ell$ are
`big'.

Before we describe the main result of this article, we recall  that if
$\rho: G_F\longrightarrow GL_N(A)$ is a representation then
 $\Ad\rho$ (resp. $\Ad^0\rho$) is the $A[G_F]$-module consisting of $N\times N$ matrices over $A$
(resp. $N\times N$ matrices over $A$ with trace $0$)  with the action of $g\in G_F$ on a
matrix $M$
given by $\rho(g)M\rho(g)^{-1}$. Let's  also recall that the   representation $\rho: G_F\longrightarrow GL_N(A)$ above
is said to be totally even
if the projective image of the decomposition group at each infinite place of $F$ is
trivial.

\begin{mainthm} Fix an integer $N\geq 3.$ Let $k$ be a finite field of characteristic $\ell$, and let $\rb: G_F\longrightarrow GL_N(k)$ be a continuous representation of the absolute Galois group of a number field $F$. Let $W:=W(k)$ denote the Witt
 ring of $k$, and fix a continuous character
 $\chi:G_F\longrightarrow W^\times$
lifting the determinant of $\rb$ (i.e. $\chi\pmod{\ell}=\det\rb$). Assume  that:
\begin{enumerate}
\item The image of $\rb$ contains $SL_N(k)$;
\item $\rb$ is not totally even;
\item If $v$ is a place of $F$ lying above $\ell$ then
$H^0\left(G_{F_v},\Ad^0\rb(1)\right)=(0).$
\end{enumerate}
Suppose that  $\ell>N^{3[F:\Q]N}$. There then exists a global deformation condition
$\cD$
with determinant $\chi$ for $\rb$  such that the universal deformation ring
for type
$\cD$ deformations of $\rb$ is a power series ring over $W$ in at
least $N-2$
variables.  In particular, there is a continuous representation
$\rho:G_F\longrightarrow GL_N(W)$
with determinant $\chi$ satisfying the following properties:
\begin{itemize}
\item $\rho \pmod{\ell}\sim \rb$; and, \item $\rho$ is unramified outside
finitely many
primes.
\end{itemize}
\end{mainthm}
We can remove the local hypothesis at $\ell$ and say more when the number field is $Q$
and $N=3$.  More precisely, let
 $\rb:G_\Q\longrightarrow GL_3(k)$ satisfy the first two conditions of the main
 theorem (so $\rb$ is odd and its image contains $SL_3(k)$).
Then $\rb$ has a lifting to $GL_3(W(k))$ whenever $\ell\geq 13$ or $\ell=7$ and  the
fixed field of $\Ad^0\rb$ does not
contain $\cos(2\pi/7).$ See Theorem~\ref{lifting GL3}.

Essentially, the claim made above is that a  residual Galois representation
with big image (including the assumption that $\ell$ is large) and good properties at
$\ell$ admits characteristic zero liftings.
We follow Mazur's development  of deformation theory as presented in \cite{mazur1};
a brief working recall of the main definitions is given in section \ref{prelims}.

The basic organisational principle underlying our approach is a beautiful result of
B\"{o}ckle,
\cite[Theorem 4.2]{boeckle2},
relating the structure of a universal deformation ring to its local (uni)versal
components. For a  precise statement see Theorem
\ref{thmboeckle} in section \ref{prelim global}. The problem thus becomes one of
finding a global deformation condition with smooth local components and trivial dual
Selmer group.
It is perhaps worth noting here that
the two
requirements are not completely independent of each other (as can be seen from the
discussion in section \ref{prelim global}).
Ramakrishna's great insight, in the $GL_2$ case, is to show how to reduce
the size of the dual Selmer group by a clever tweaking of the global deformation
condition at some primes. We extend this strategy.

There are two key
ingredients in being able to make such an extension. Firstly, we prove a
cohomological
result which gives conditions under which a subspace of $H^1(G_F,M)$ can be
distinguished by its
restriction at a prime. This provides us with a collection of primes where an
adjustment of the local condition can result in a smaller dual Selmer group.
Secondly, we need to produce enough local deformations  for the
restriction of
$\rb$ to a local decomposition group at a prime  $v\nmid\ell$. There are complications
when the residue characteristic of $F_v$ is relatively small (for instance, when the residue
characteristic is not
bigger than $N$), and we avoid these by assuming
$[F_v(\zeta_\ell):F_v]\geq 3N$. (See Theorem \ref{main local}.) The  condition
 $\ell>N^{3[F:\Q]N}$ is an easy---but not an economic---bound that allows us to avoid local complications at small primes for general $N$, $\ell$.

While the hypothesis at  primes above $\ell$ ensures that we do not have to deal with
the
more difficult problem of studying local deformations at $\ell$, it does still cover  a
wide range of examples. Note that  the
hypothesis at a prime $v|\ell$ is equivalent to the assumption that the only
$G_{F_v}$-equivariant homomorphism from $\rb$ to $\rb(1)$ is the zero map.
The exceptions can be easily classified for small $N$, and we do so for the case
when $N=3$ and $F=\Q$. We do not attempt to
put any geometric condition as the representations we are looking at might not
even have the right duality property (to link up with automorphic forms).

A similar generalisation of Ramakrishna's lifting technique to $GL_N$ was also
obtained by Hamblen,  \cite{hamblen}, about the same time when an earlier version of
this article was first prepared. Even so, we hope that this article  still carries an
interest. For one, the
results
are different  (Hamblen assumes the ground field to be $\Q$ and uses different local
conditions). Additionally, we hope that the study of local deformations presented here, in particular the existence of smooth deformations, has independent merit. Although some of the local analysis also appears in \cite{CHT}, there is a difference in approach (for instance in the study of tamely ramified deformations and also in the role of tensor product of deformations). 

\begin{notation}
The $\ell$-adic cyclotomic character is always denoted by $\omega$ and $\overline\omega$ is the
mod $\ell$-cyclotomic character. The term
`prime' on its own  always indicates a finite prime except when the context makes it
clear that we are also including infinte primes. If $F$ is a number field, we assume we are given fixed
embeddings $\overline F\hookrightarrow \overline F_v$ for each prime $v$ (including the infinite ones).  If $F$ is unramified at
the prime $v$ we shall view $\Frob_v$ as element of $G_F$ via the embedding
$\overline
F\hookrightarrow \overline F_v$. If $A$ is a topological ring and
$\rho:G_F\longrightarrow
GL_N(A)$ is a continuous representation, we shall denote the restriction of $\rho$ to a
decomposition group at $v$ by $\rho_v.$ We shall frequently use $H^*(F,M)$ to denote
$H^*(G_F,M)$. The group of unramified cohomology classes at a prime is indicated by the
presence of a subscript (as in
$H^*_{\textrm{nr}}$).

If $k$ is a finite field then the Witt ring of $k$ will be denoted by $W(k)$ and
$\widehat{x}\in W(k)$
denotes the  Teichm\"{u}ller lift of $x\in k.$ A CNL $W(k)$-algebra, or simply a CNL
algebra if the finite field $k$ is clear,  is shorthand for a complete, Noetherian, local algebra {\bf with} residue field $k.$  If
$\chi$ (resp. $\rho$) is a $W(k)$ valued character (resp. homomorphism) then we will use
the same letters for their extension to a CNL $W(k)$-algebra.
\end{notation}

\section{Preliminaries}
\label{prelims}

We now give a brief summary of deformation theory and discuss some of the key tools used  in studying global deformation conditions. Aside from setting out notation, we hope that the discussion in this section (following Theorem \ref{thmboeckle} in particular) will make transparent the basic argument and structure of this article.

\subsection{Deformation conditions in general} \label{gen DC}

We begin with a brief summary of what a deformation condition means since, for the most part, we
shall be involved in checking that the properties we specify at a local decomposition group give
a deformation condition. We shall follow \S23, \S26 of \cite{mazur1}, except for some minor
adjustments.

Let  $\Pi$ be a profinite group satisfying the
 ``finiteness at $\ell$'' property of Mazur (\S1 of \cite{mazur1}). For our purposes, a
 representation of
 $\Pi$ is a continuous homomorphism $\rho:\Pi\longrightarrow GL_N(A)$ where $A$ is a
 topological
 ring. The underlying free $A$-module on which $\Pi$ acts will be denoted by $V(\rho).$
 Given two representations
 \[
 \rho_A:\Pi\longrightarrow GL_N(A),\qquad \rho_B:\Pi\longrightarrow GL_N(B)
 \]
  and a morphism
 $f:A\longrightarrow B$ in the relevant category, we say that $\rho_A$ is
 a lift of $\rho_B$ if $f\rho_A=\rho_B$.

 If $\rho_1:\Pi\longrightarrow GL_n(A)$, $\rho_2:\Pi\longrightarrow GL_m(A)$ are two representations then
 $\Hom (V(\rho_1),V(\rho_2))$, or just simply $\Hom (\rho_1,\rho_2)$, is shorthand
 for the $A[\Pi]$-module of $A$-linear maps from $V(\rho_1)$ to $V(\rho_2)$. As a
 representation $\Hom (\rho_1,\rho_2)$ can be described as the group of $m\times n$ matrices over $A$  with $\Pi$ action given by
 $(g,M)\longrightarrow\rho_2(g)M\rho_1(g)^{-1}$.  We shall take $\rho_1\otimes\rho_2:G_F\longrightarrow
GL_{mn}(A)$ to mean the representation (gotten from $V(\rho_1)\otimes V(\rho_2)$) expressed with respect to the basis $v_1\otimes
w_1,\ldots, v_1\otimes w_{m},\ldots , v_{n}\otimes w_1,\ldots ,v_{n}\otimes w_{m}$ where
$v_1,\ldots, v_{n}$ and $w_1,\ldots , w_{m}$ are the bases for $\rho_1$ and $\rho_2$
respectively. Note that $\Hom(\rho_1,\rho_2)$ is naturally isomorphic to
$\rho_{1}^{*}\otimes\rho_2$ where $\rho_{1}^{*}$ is the dual representation for $\rho_1$.

Let $\rep_N(\Pi;k)$ denote the following category:
\begin{itemize} \item Objects are pairs $(A,\rho_A)$ where $A$ is a CNL $W(k)$-algebra
and
$\rho_A:\Pi\longrightarrow GL_N(A)$ is a  representation; \item A morphism from $(A,\rho_A)$ to
$(B,\rho_B)$ is a pair $(f,M)$ where $f:A\longrightarrow B$ is a morphism of local rings and
$M\in GL_N(B)$ satisfies $f\rho_A=M\rho_BM^{-1}.$
\end{itemize}
Given a representation
$\rb:\Pi\longrightarrow GL_N(k),$
a \emph{deformation condition} $\cD$ for $\rb$ is
 a full subcategory $\cD\subseteq\rep_N(\Pi;k)$  satisfying the following properties:
\begin{enumerate}

\item[(DC0)] $(k,\rb)\in \cD,$ and if $(A,\rho_A)\in \cD$ then
$\rb\sim\rho_A\mod{\m_A}$.

\item[(DC1)] If $(A,\rho_A)$ is an object in $\cD$
and $(f,M):(A,\rho_A)\longrightarrow (B,\rho_B)$ is a morphism, then $(B,\rho_B)$ is also in
$\cD$.

\item[(DC2)] Given a cartesian diagram
\[
\begin{CD}
A\times_CB @>\pi_B>> B\\
@V{\pi_A}VV @V{\beta}VV\\
A @>{\alpha}>> C
\end{CD}
\]
of Artinian CNL algebras $A,B,C$ with $\beta$ small, an object
$(A\times_CB,\rho)$ of $\rep_N(\Pi;k)$ is in $\cD$ if and only if $(A,\pi_A\rho),
(B,\pi_B\rho)$ are in $\cD.$

\end{enumerate}
We say that $\rho:\Pi\longrightarrow GL_N(A),$ or $(A,\rho),$ is of type $\cD$ if
$(A,\rho)$ is in $\cD.$ If $\chi:\Pi\longrightarrow W^\times$ is a character, we say that
$\cD$ has determinant $\chi$ if $\det\rho=\chi$ for any $(A,\rho)\in\cD.$
The deformation condition $\cD$ is said to be smooth if for any surjection
$f:A\longrightarrow B$ and an object $(B,\rho_B)$ of type $\cD,$ there is an
object
$(A,\rho_A)$ in $\cD$ such that $f\rho_A=\rho_B.$ It is sufficient to verify
the
smoothness condition for small extensions only. The tangent space of $\cD$ will
be
denoted by $T\cD,$ and will be viewed as a $k$-subspace of
$H^1(\Pi,\Ad\rb)$
(it is a subspace of $H^1(\Pi,\Ad^0\rb)$ if the determinant is fixed).

In practice, conditions (DC0), (DC1), and  the only if part of condition (DC2), will
almost
always be immediate. If $\cD$ is a deformation condition for
$\rb:\Pi\longrightarrow GL_N(k),$ the functor
\[
\cD(A):=\left\{\text{type $\cD$ liftings}\quad \rho:\Pi\longrightarrow GL_N(A)\quad
\text{of}\quad \rb\right\}/\text{strict equivalence}
\]
 is nearly representable. If
$\cD$ is smooth then the (uni)versal deformation ring is a power series ring.

Our objective is to produce (uni)versal deformation rings which are power series rings.
In view
of the following lemma, one can make use of extension of scalars to produce such (uni)
versal
deformation rings.

\begin{lemma}
\label{scalar extn}
Let $k_0\subset k_1$ be finite fields of characteristic $\ell,$ and let
$\rb_0:\Pi\longrightarrow GL_n(k_0)$ be a representation. Denote by
$\rb_1:\Pi\longrightarrow GL_n(k_1)$ the extension of scalars of $\rb_0$ to
$GL_n(k_1).$

Given a deformation condition $\cD_1\subseteq\rep_n(\Pi;k_1),$ let $\cD_0$ be the
full subcategory of $\rep_n(\Pi;k_0)$ consisting of those objects $(A,\rho)\in\rep_n(\Pi;k_0)$
such that $\left(A\otimes_{W(k_0)}W(k_1),\rho\otimes W(k_1)\right)\in\cD_1.$ Then:
\begin{enumerate}
\item $\cD_0$ is a deformation condition for $\rb_0,$ and
$\dim_{k_0}T\cD_0=\dim_{k_1}T\cD_1.$ \item Let $R_0,R_1$ be the (uni)versal
deformation rings of type $\cD_0,\cD_1.$ Then there is an isomorphism
$R_1\longrightarrow R_0\otimes_{W(k_0)}W(k_1).$ In particular, if $R_1$ is a power series ring
then so is $R_0.$\end{enumerate}
\end{lemma}

\begin{proof} Checking that $\cD_0$ is a deformation condition is straightforward.
Extension of scalars give a natural isomorphism between $H^1(\Pi,\Ad\rb_0)\otimes k_1$
and $H^1(\Pi,\Ad\rb_1).$ Thus there is a subspace $L\subseteq
H^1(\Pi,\Ad\rb_0)$ such that $L\otimes k_1=T\cD_1.$ One then checks that $L$
has to be the tangent space for $\cD_0.$

For the second part, there is a surjection $R_1\longrightarrow R_0\otimes W(k_1).$ Since the
extension $W(k_1)/W(k_0)$ is smooth, the tangent space for $R_0\otimes W(k_1)$ has the same
dimension as the tangent space for $R_0.$ Hence the surjection is an isomorphism. \end{proof}

\subsection{Global deformations}
\label{prelim global}

Now let $F$ be  a number field and  let $k$ be a finite field of characteristic
$\ell$. Fix an absolutely irreducible representation
$\rb:G_F\longrightarrow GL_N(k)$ and a character
 $\chi:G_F\longrightarrow W^\times$  such that
$\chi\pmod{\ell}=\det\rb$.

Informally, a global deformation condition
specifies that we consider liftings of $\rb:G_F\longrightarrow GL_N(k)$ with
prescribed local behaviour.
More precisely: Suppose we are given, for each prime $v$ of $F$,  a
deformation condition $\cD_v$ for $\rb\vert_v$ with determinant $\chi$. Furthermore,
we require that the deformation condition $\cD_v$
is unramified for  almost all primes $v$.  The global deformation condition $\{\cD_v\}$
with determinant $\chi$ for $\rb$ is  then the full subcategory of $\rep_N(G_F;k)$
consisting
of those objects $(A,\rho)\in\rep_N(G_F;k)$ such that $\det\rho=\chi$ and
$(A,\rho|_v)\in\cD_v$ for all $v$.

For a global deformation condition $\cD$ with determinant $\chi$ for $\rb$, we
shall denote the local condition at $v$ by $\cD_v$ (so $\cD=\{\cD_v\})$.
We define the ramification set $\Sigma(\cD)$ to be the finite set consisting of those
primes $v$ of $F$ where $\cD_v$ is not unramified, primes  lying above $\ell$ and
$\infty$, and primes where $\rb$ and $\chi$ are ramified. Thus $\cD$ is
precisely a deformation condition for $\rb|_{\textrm{Gal}(F_{\Sigma(\cD)}/F)}$
with prescribed local components (cf. \S26 of \cite{mazur1}). The tangent space for
$\cD$
is the Selmer group
\[
H^1_{\{T\cD_v\}}\left(F,\Ad^0\rb\right)=
\textrm{ker}\left(H^1(G_F,\Ad^0\rb)\longrightarrow\prod
H^1(F_v,\Ad^0\rb)/T\cD_v\right).
\]
The dual Selmer group for $\cD$ is defined as follows. For each prime $v$ of $F$ the  pairing
$\Ad^0\rb\times \Ad^0(1)\rb\longrightarrow k(1)$ obtained by taking trace induces a perfect pairing
\[
H^1\left(F_v,\Ad^0\rb\right)\times H^1\left(F_v,\Ad^0\rb(1)\right)
\longrightarrow
H^2(F_v,k(1)).
\]
Let $T\cD_v^\perp\subseteq
H^1\left(F_v,\Ad^0\rb(1)\right)$ be the annihilator of
$T\cD_v$ under
the above pairing. The dual Selmer group
$H^1_{\{T\cD_v^\perp\}}\left(F,\Ad^0\rb(1)\right)$ is then determined by the local conditions
$\{T\cD_v^\perp\}$ i.e.
\[
H^1_{\{T\cD_v^\perp\}}\left(F,\Ad^0\rb(1)\right):=
\textrm{ker}\left(H^1(G_F,\Ad^0\rb(1))\longrightarrow\prod
H^1(F_v,\Ad^0\rb(1))/T\cD_v^\perp\right).
\]
While the tangent space for $\cD$ is a very difficult object to get a handle on, remarkably a quantitative comparision with the dual Selmer group is possible by the following formula of Wiles (Theorem 8.6.20 in \cite{NSW}):
\begin{multline}\label{wiles formula}
\dim H^1_{\{T\cD_v\}}\left(F,\Ad^0\rb\right)-\dim H^1_{\{T\cD_v^\perp\}}\left(F,\Ad^0\rb(1)\right) \\
= \sum_{\text{primes}\ v\leq \infty}\left(\dim T\cD_v-\dim H^0(F_v, \Ad^0\rb)\right).
\end{multline}

We now describe a beautiful result of B\"{o}ckle which allows one to relate the global (uni)versal
deformation ring in terms of local deformation rings. Let  $\overline\rho$, $\chi$ and $\cD$  be as above. For each prime $v$, set
$n_v:=\dim T\cD_v$.
The (uni)versal deformation ring $R_v$ for type $\cD_v$ deformations of
$\overline\rho|_{G_{F_v}}$ then has a presentation
$R_v\cong
W(k)[[T_{v,1},\ldots , T_{v,n_v}]]/J_v$.
Note that $J_v=(0)$ if $v\not\in\Sigma(\cD)$.

 Fix  a presentation
\[
W(k)[[T_1,\ldots ,T_n]]/J ,\qquad  n:=H^1_{\{T\cD_v\}}(F,\Ad^0\overline\rho),
\]
for the (uni)versal global deformation ring $R$ for type $\cD$ deformations of
$\rb$. Restriction of the (uni)versal deformation to a decomposition group at $v$ induces a map  $R_v\longrightarrow R$ which can be then lifted to a map
$\alpha_v:W(k)[[T_{v,i}]]\longrightarrow W(k)[[T_i]]$ of local rings. (Of course
$\alpha_v$, and even $R_v\longrightarrow R$, might not be unique at all.)

 \begin{theorem} (B\"{o}ckle, Theorem 4.2 of \cite{boeckle2})
  With notation as in the preceding paragraphs, the ideal $J$ is generated
  by the images $\alpha_vJ_v$ together with at most
  $\dim  H^1_{\{T\cD_v^\perp\}}(F,\Ad^0\overline\rho(1))$ other elements. Thus
\begin{align}\label{loc global inequality}
 \gen(J)\leq \sum_{v\in \Sigma(\cD)}\gen(J_v) +
  \dim H^1_{\{T\cD_v^\perp\}}(F,\Ad^0\overline\rho(1))
\end{align}
where $\gen(J)$ (resp. $\gen(J_v)$) is the minimal number of elements  required to generate  the ideal $J$ (resp. $J_v$).
  \label{thmboeckle}
\end{theorem}

To prove our main theorem, we make sure that our global deformation condition has
smooth local conditions and trivial dual Selmer group. If we can do that, then
\eqref{loc global inequality} ensures that the global deformation ring has trivial
ideal of relations and so is smooth. The question now is how to get to such nice global
deformations.

The first step is to construct a global deformation problem $\cD$ with smooth local
deformation conditions. By \eqref{loc global inequality} the number of global relations
is then bounded by the dimension of the dual Selmer group. The next, and critical step,
is to tweak one of the local conditions $\cD_v$ at some prime so that the new
deformation condition has smaller dual Selmer group. We shall show that this can be
done provided
\begin{align}\label{ineq 1}
\dim H^1_{\{T\cD_v\}}\left(F,\Ad^0\rb\right)\geq N-2 +
\dim H^1_{\{T\cD_v^\perp\}}\left(F,\Ad^0\rb(1)\right).
\end{align}
We shall prove the necessary results from Galois cohomology in Section \ref{Gal coh}.

Note that by  Wiles formula \eqref{wiles formula}, the above inequality will fail if
the local deformation conditions are `small'. To ensure this doesn't happen, we make
sure that $\cD_v$ is smooth in $\dim H^0(F_v,\Ad^0\rb)$ variables at primes not
dividing $\ell$. The required constructions are carried out in Section
\ref{local analysis}; the precise statement we need is presented in Theorem~\ref{main local}. Given these local conditions,  the hypotheses at $\ell$ and
 $\infty$ allows us to ensure that \eqref{ineq 1} is satisfied.

\section{Galois cohomology}\label{Gal coh}

\emph{Throughout this section,  $K/F$ is a finite Galois extension of number fields with
Galois group
$G:=\textrm{Gal}(K/F)$ and  $k$ be a finite extension of $\F_\ell$.}

If $M$ is   a  $k[G]$-module and $\xi\in H^1(G_F,~M)$ then the restriction
 of $\xi$ to $G_K$ is a group homomorphism. We denote by $K(\xi)$ the field through which this
homomorphism factorises. Note that the extension $K(\xi)/F$ is Galois. For
$\xi_i\in H^1(G_F,~M)$,
$i=1,\ldots , n$, the compositum of $K(\xi_1),\ldots , K(\xi_n)$ will be denoted by
$K(\xi_1,\ldots ,\xi_n)$.

\subsection{} In this subsection $M$ is a  finite $k[G]$-module  satisfying the following
two conditions:
\begin{itemize}
\item $M$ is a simple $\F_\ell[G_F]$-module with $\textrm{End}_{\F_\ell[G_F]}(M)=k$; \item $H^1(G,~M)=0$.
\end{itemize}

\begin{lemma} \label{coh-lemma}
Let $0\neq\xi\in H^1(G_F,~M).$  Then:
\begin{enumerate}
\item[(a)] The restriction $\xi : \textrm{Gal}\left(K(\xi)/K\right)\longrightarrow M$ is an
isomorphism of $G$-modules. \item[(b)] If $L$ is a Galois extension of $F$ with $K\subseteq M,$
then  either $K(\xi)\subseteq L$ or $K(\xi)\cap L=K.$
\end{enumerate}\end{lemma}

\begin{proof} The images of $\textrm{Gal}(K(\xi)/K)$ and
$\textrm{Gal}(K(\xi)/(K(\xi)\cap L)$ under $\xi$ are subspaces of $M$ stable under the action of
$G.$ The lemma follows as $M$ is simple.\end{proof}

\begin{proposition}\label{disjoint}
 If $\psi_1,\psi_2 ,\ldots ,\psi_n$ are $n$ linearly
independent classes in the $k$-vector space $H^1(G_F,~M),$ then $K(\psi_1), K(\psi_2),\ldots ,
K(\psi_n)$ are linearly disjoint over $K.$
\end{proposition}

\begin{proof} We first do the case $n=2.$ If $K(\psi_1)$ and
$K(\psi_2)$ are not linearly disjoint over $K,$ then by the above lemma $K(\psi_1)=K(\psi_2).$
The composite
\[
M\xrightarrow{\psi_1^{-1}}
\textrm{Gal}(K(\psi_1)/K)=\textrm{Gal}(K(\psi_2)/K)\xrightarrow{\psi_2}M
\]
is a
$G$-module automorphism of $M$. Since $k$ is the endomorphism ring of $M$,  $\psi_1$ and $\psi_2$
are linearly dependent---a contradiction.

Suppose now that the proposition holds for $n-1.$ Let $K(\psi_1,\ldots , \psi_{n-1})$ be the
compositum of $K(\psi_1),\ldots , K(\psi_n).$ By the inductive hypothesis, we have
identifications
\[
\begin{matrix}
\textrm{Gal}\left(K(\psi_1,\ldots , \psi_{n-1})/K\right)&\longrightarrow&
\textrm{Gal}\left(K(\psi_1)/K\right)\times\cdots\times
\textrm{Gal}\left(K(\psi_{n-1})/K\right)\\
&& \downarrow \\
&&M\times\cdots\times M\end{matrix}
\]
 of $G$-modules.

Let $V$ be the $k$-subspace of $H^1(G_F,~M)$ spanned by $\psi_1,\ldots , \psi_{n-1}$ and let
$\mathcal{E}$  be the set Galois extensions $E/F$ with $K\subseteq E\subseteq K(\psi_1,\ldots ,
\psi_{n-1})$ and $\textrm{Gal}(E/K)$ isomorphic to $M$ as $G$ modules. We claim that the map
$\mathbb{P}(V)\longrightarrow \mathcal{E}$ given by $\psi\longrightarrow K(\psi)$ is a bijection.
That the map is an injection   follows from the case $n=2$ of the proposition. Now elements of
$\mathcal{E}$ correspond to non-trivial $G$ module homomorphisms from $M\times\cdots \times M$ to
$M.$ Now
\begin{align*}
\Hom_G\left(M\times\cdots\times M,~M\right) &\cong
\Hom_G\left(M,~M\right)\times\cdots\times
\Hom_G\left(M,~M\right)\\
&\cong k\times\cdots\times k.\end{align*} It follows that $|\mathbb{P}(V)|=|\mathcal{K}|,$
and this establishes the claim.

We now show that $K(\psi_n)$ and $K(\psi_1,\ldots , \psi_{n-1})$ are linearly disjoint over $K.$
If not, then Lemma~\ref{coh-lemma} implies that $K(\psi_n)\in \mathcal{E},$ and so by the claim
$K(\psi_n)=K(a_1\psi_1+\cdots +a_{n-1}\psi_{n-1})$ for some $a_1,\ldots ,a_n\in k.$  Appealing to
the case $n=2$ of the proposition, we see that $\psi_n$ is a linear combination of
$\psi_1,\ldots, \psi_{n-1}$---which is a contradiction.
\end{proof}

\subsection{} Now let $M$ be an absolutely irreducible $k[G]$-module with
$H^1(G,~M)=0,$ and let $g\in G$ be a fixed element of $G$ which acts semi-simply on $M.$ We
denote by $M^g$ the kernel of multiplication by $g-1$ on $M.$ Note that we have a decomposition
$M=M^{g}\oplus (g-1)M.$

 Fix also a
non-trivial subgroup $L\subseteq M$ invariant under $G_F$ with minimal dimension as an
$\F_\ell$-vector space. It is then straightforward to check that $L$ is simple, that $k$ contains
$ \textrm{End}_{\F_\ell[G_F]}(L)=:k'$ (say), and that $M\cong L\otimes_{k'}k.$ Further, we have
$M^{g}=L^{g}\otimes_{k'}k$ and $(g-1)M=(g-1)L\otimes_{k'}k.$

\begin{proposition} \label{injective1}
With assumptions and notations as in the previous two paragraphs,
let $V$ be a finite dimensional $k$-subspace of  $H^1(G_F,~M).$ If
$\dim M^g\geq
\dim V$ we can find a lift $\tilde{g}\in G_F$ of $g$ such that the restriction map
\[
V\hookrightarrow H^1(G_F,~M)\longrightarrow H^1(\langle \widetilde{g}\rangle,~M)
\]
 is
injective.
\end{proposition}

\begin{proof} Set $n:=\dim V$. Since $H^1(G_F,~M)\cong H^1(G_F,~L)\otimes_{k'}k$,
we can find:
\begin{itemize}
\item a basis $\xi_1,\ldots, \xi_n$ of $V$,
 \item $m$ linearly independent cocyles $\psi_1,
\ldots, \psi_m$  in the $k'$-vector space $H^1(G_F,~L)$ with $m\geq n$ and such that
$\xi_i:=\psi_i +\sum_{j>n}a_{ij}\psi_j$
for some $ a_{ij}\in k,\ i=1,\ldots , n.$ \end{itemize}

Fix a lift $g'\in G_F$ of $g.$ We can identify $H^1(\langle g'\rangle ,~M)$ with $M^g.$  For ease
of notation, we set
\[
K_0:=K(\psi_j,j>n),\ \textrm{and}\ K_i:=K(\psi_i,\psi_j,j>n),\
i=1,\ldots, n.
\]
 By Proposition~\ref{disjoint}, the extensions $K_i, \ i=1,\ldots ,n $ are linearly
disjoint over $K_0.$

For each $1\leq i\leq n,$ the cocyle $\xi_i$ restricts to $\psi_i$ on $K_0.$ Since
$\psi_i(\textrm{Gal}(K_i/K_0)=L$ and $\xi_i(xg')=\psi_i(x)+\xi_i(g')$ for any
$x\in\textrm{Gal}(K_i/K_0),$ we see that the $k$-subspace of $M$ generated by $\psi_i(xg')$ is
$M$.

We claim that we can find $x_i\in\textrm{Gal}(K_i/K_0),\ 1\leq i\leq n,$ such that
$\xi_1(x_1g'),\ldots,\xi_n(x_ng')$ generate an
 $n$-dimensional subspace of $M/(g-1)M.$ To see this, first pick
$x_1\in\textrm{Gal}(K_1/K_0)$ such that $\xi_1(x_1g')$ is non-trivial when projected to
$M/(g-1)M.$ Having found $x_i\in\textrm{Gal}(K_i/K_0),\ i=1,\ldots , j$ with $j<n$ and such that
$\xi_1(x_1g'),\ldots, \xi_j(x_jg')$ generate a
 $j$-dimensional subspace of $M/(g-1)M$ we can find an
 $x_{j+1}\in\textrm{Gal}(K_{j+1}/K_0)$
 with the property that
 $\xi_{j+1}(x_{j+1}g')$ does not lie in the subspace of $M$ spanned by
 $\xi_1(x_ig'),\ldots,\xi_j(x_jg')$ and $(g-1)M$ (possible as this latter
 subspace has dimension $j+\dim_k(g-1)M<\dim_k
 M$).

 Finally, using Proposition \ref{disjoint}, we can find $x$ in the Galois group of $K_0$
 which acts as $x_i$ on each extension
 $K_i/K_0.$ Set $\widetilde{g}=xg'.$
 Then as $\xi_1(\widetilde{g}),\ldots,\xi_n(\widetilde{g}) $ generate an $n$-dimensional subspace of
 $M/(g-1)M,$ we see that the images of $\xi_i$ when restricted to
 $H^1(\langle \widetilde{g}\rangle,~M)$ are linearly independent.
\end{proof}

\begin{theorem} \label{Gal coh thm}
Let $M_1,\ldots , M_n$ be $n$ inequivalent, absolutely irreducible $k[G]$ modules with
$H^1(G,~M_i)=0,\ 1\leq i\leq n.$ We assume that we are given a place $v$ of $F$ and $k$-subspaces
$V_i\subseteq H^1(G_F,~M_i)$ with the following properties:
\begin{itemize}
\item  $M_1\oplus\cdots\oplus M_n$ is unramified at $v$, and
that $\Frob_v$ acts semi-simply on each $M_i$;
\item  $\dim V_i\leq
\dim H^1_\textrm{nr}(F_v,~M_i)$ for $i=1,\ldots, n.$
\end{itemize}

Under the above assumptions, we can find infinitely many places $w$ such that:
\begin{itemize}
\item $M_1\oplus\cdots\oplus M_n$ is unramified at $w$ and the images of $\Frob_w,\Frob_v$ in $G$
are the same;
\item Any cohomology class in $V_i$ is unramified at $w$;
 \item The restriction
map
$V_1\oplus\cdots\oplus V_n\longrightarrow H^1_\textrm{nr}(F_w,~M_1)\oplus\cdots\oplus
H^1_\textrm{nr}(F_w,~M_n)$ is injective.
\end{itemize}
\end{theorem}

\begin{proof} Denote by $K(V_i)$ the splitting field for $V_i$
over $K,$ and by $K(V_1,\ldots ,V_n)$ the compositum of $K(V_i)$. The extensions
$K(V_i)$ are
linearly disjoint over $K$ because $\textrm{Gal}(K(V_i)/K)$ is isomorphic to a subgroup
of $M_i.$

Take $g\in G$ to be an element which $\Frob_v$ lifts and let $g'\in
\textrm{Gal}(K(V_1,\ldots ,V_n)/F)$ be a lift of $g$.
 By Proposition~\ref{injective1}, we can find
$x_i\in\textrm{Gal}(K(V_i)/K)$ such that
$V_i\longrightarrow H^1(\langle x_ig'\rangle ,~M_i)$ is
injective. Using disjointness of the $K(V_i)$'s, we can find an
$x\in \textrm{Gal}(K(V_1,\ldots,V_n)/K)$ such that $x$ acts on $K(V_i)$ as $x_i.$ By
the Chebotarev density theorem, we can then find
a place $w$ of $F$ lifting $xg'$ and unramified in $K(V_1,\ldots ,V_n).$  It is now
immediate
such a $w$ satisfies the properties asked for.
\end{proof}

\section{Local deformation conditions } \label{local analysis}

Throughout this section $k$ is a finite field of characteristic $\ell$ and $p$ is a
prime different from $\ell$. We shall look at deformation conditions for  a finite extension of $\Q_p.$ In particular, our objective  is to construct examples of local deformation
conditions which admit a large (uni)versal deformation ring. The precise nature of
what large should mean is the content of the following definition.

\begin{definition} Let $F$ be a finite extension of $\Q_p$ and let
$\rb:G_F\longrightarrow GL_N(k)$ be a representation. We say that a deformation
$\cD$ for $\rb$ is  \emph{well-behaved} if $\cD$ is smooth and
$\dim\ T\cD=\dim \ H^0(G_F,\Ad\rb)$.
\end{definition}

\begin{example}\label{ex 1}
If $\rb$ is unramified then the class of unramified liftings is a well-behaved
deformation condition. The unrestricted deformation condition is well-behaved if
$H^2(G_F,\Ad\rb)=(0)$.
\end{example}

We can now state our
 main result asserting the existence of well-behaved deformation conditions.

\begin{theorem} \label{main local}
Let $F$ be a finite extension of $\Q_p,$  let
$k$ be a finite field of characteristic $\ell\neq p$, and let
$\rb:G_F\longrightarrow GL_N(k)$ be a representation. Assume that all irreducible
components occurring in the semi-simplification of $\rb$ are absolutely irreducible. If
$p\leq N$ and $\rb$ is wildly ramified assume that $[F(\zeta_\ell):F]\geq 3N$ where $\zeta_\ell$ is an $\ell$-th root of
unity. Then the following hold:
\begin{enumerate}
\item[(a)] There  exists a well-behaved deformation condition $\cD$.
 \item[(b)] Suppose $\chi:G_F\longrightarrow W^\times$ is a character lifting
 $\det\rb$. Assume that $N,\ell$ are co-prime. Then liftings of type $\cD$
and determinant $\chi$ is a smooth deformation condition for $\rb$ and the dimension
of its  tangent is equal to
$\dim H^0(G_F,\Ad^0\rb).$
\end{enumerate}
\end{theorem}

To construct a well-behaved deformation condition $\cD$ as claimed (and also to outline the structure of this section), we proceed as follows:
\begin{itemize}
\item We would like to build up $\cD$ from well-behaved deformation conditions for
some decomposition of $\rb$. In section \ref{products} we show that a good way of
decomposing  $\rb$ is to make sure that the basic blocks have no common irreducible components, even after taking Tate twists. (See relation \ref{prod relation}.)
\item The blocks can then be analysed separately. There are essentially three cases we
need to consider.
\begin{itemize}
\item Firstly, the case when a given residual representation is tamely ramified. The
deformation condition in this case is to obtained by specifying a Jordan--Holder
decomposition for a generator of tame inertia. See section \ref{tamely ramified}.
\item The residual representation is a tensor product of two smaller representations.
In section \ref{tensor prod} we study when we can construct the candidate well-behaved
deformation by using tensor products.
\item The residual representation is induced, in which case we try to induce a known
well-behaved deformation condition. This is done in section \ref{induced repn}
\end{itemize}
\item Finally, we verify that the hypotheses of Theorem \ref{main local} guarantee
applicability of the preceding steps and complete the proof in section
 \ref{main local proof}.
\end{itemize}

The second part of Theorem \ref{main local} is straightforward, and we deal with it
right away:

\begin{proof}[Proof of Theorem \ref{main local} (b)]
 We need only check smoothness, and for that if suffices to check that any
deformation $\rho:G_F\longrightarrow GL_N(A)$ of type $\cD$ can be twisted to a
deformation with determinant $\chi$. If $\psi: G_F\longrightarrow A^\times$
is a character and we want $\chi=\det(\psi\rho)$, then  $\psi^N=\chi\det\rho^{-1}$.
We can find such a character $\psi$ because
$\chi\det\rho^{-1}:G_F\longrightarrow
1+\mathfrak{m}_A$
and
\[
\begin{CD}
1+\mathfrak{m}_A@>{x\rightarrow x^N}>>1+\mathfrak{m}_A
\end{CD}
\]
is an isomorphism.\end{proof}

\subsection{Products of deformation conditions}
\label{products}
Let $F$ be a finite extension of $\Q_p$. We assume  we are given
  representations
$\rb_i:G_F\longrightarrow GL_{d_i}(k), \ i=1,\ldots ,
n$ satistfying
\begin{align} \label{prod relation}
\Hom_{k[G_F]}\left(\rb_i,\rb_j(r)\right)=(0)
\end{align}
for $i\neq j, r\in\Z$, and a deformation condition $\cF_i$ for $\rb_i$, $i=1,\ldots ,n$.
 Set $\rb:=\rb_1\oplus\cdots\oplus\rb_n$ and $N:=d_1+\ldots+d_n$.

We shall say that a representation $\rho:G_F\longrightarrow GL_N(A)$  is of type
$\cF:=\cF_1\oplus\cdots\oplus\cF_n$ if $\rho\sim\rho_1\oplus\cdots\oplus\rho_n$ with
$(A,\rho_i)\in\cF_i.$ We denote by $\cF$
the full subcategory of $\rep_N(G_F;k)$ consisting of objects $(A,\rho)$ with $\rho$ of type
$\cF_1\oplus\cdots\oplus\cF_n.$ We then have the following

\begin{theorem}
$\cF$ is a deformation condition for $\rb.$ The natural map
\[
\left((A,\rho_i)\in
\cF_i\right)_{i=1}^n\longrightarrow (A,\rho_1\oplus\cdots\oplus\rho_n)
\]
induces an
isomorphism of tangent spaces
\[
T\cF\cong
T{\cF_1}\oplus\cdots\oplus T{\cF_n},
\]
and  $\cF$ is
well-behaved if each $\cF_i$ is well-behaved.\label{directsum}
\end{theorem}

Theorem \ref{directsum} is an immediate consequence of the following proposition:
\begin{proposition} Let $R$ be a CNL algebra, and let
$\rho:G_F\longrightarrow GL_N(R)$ be a lift of $\rb.$ We then have, up to  strictly
equivalence, a unique decomposition $\rho\cong\rho_1\oplus\cdots \oplus \rho_n$ where $\rho_i:G_F\longrightarrow GL_{d_i}(R)$ is
a lift of $\rb_i.$\label{H1}\end{proposition}

The proof of Proposition \ref{H1} relies on there being no cohomological relations between lifts of $\rb_i$ and $\rb_j$ when $i\neq j$. More precisely, we need the following lemma:
\begin{lemma}
Let $*=0$, $1$ or $2$.
\begin{enumerate}
\item If $i\neq j$ then $H^*\left(G_F,\Hom(\rb_i,\rb_j)\right)=(0)$ if $i\neq j$. Consequently, we have
\begin{align*}
H^*\left(G_F,\Hom(\rb,\rb)\right)
\cong H^*\left(G_F,\Hom(\rb_1,\rb_1\right)\oplus\cdots\oplus
H^*\left(G_F,\Hom(\rb_n, \rb_n)\right).
\end{align*}
\item  Let $A$ be an Artinian CNL algebra, and let
$\rho_i:G_F\longrightarrow GL_{d_i}(A)$, $\rho_j:G_F\longrightarrow GL_{d_j}(A)$ be lifts of
$\rb_i, \rb_j, i\neq j.$ Then
\[
H^*(G_F,\Hom(\rho_i, \rho_j))=(0).
\]
\end{enumerate}
\label{H^i}
\end{lemma}

\begin{proof} The first part follows easily from relation \ref{prod relation}, local duality and the local Euler characteristic formula.

For the second part, let $J$ is an ideal of $A$ with
$\mathfrak{m}_AJ=(0)$. Then
\[
0\longrightarrow \Hom(\rho_i,\rho_j)\otimes J
\longrightarrow\Hom(\rho_i,\rho_j)\longrightarrow \Hom(\rho_i\,\textrm{mod}\,
J,\rho_j\,\textrm{mod}\, J)\longrightarrow 0
\]
 is an exact sequence of $G_F$-modules. Induction
along with the first part then completes the proof.\end{proof}

\begin{proof}[Proof of Proposition \ref{H1}] We can take $R$ to be Artinian. Let $\mathfrak{m}$
be its maximal ideal, and let $J\neq (0)$ be an ideal of $R$ killed by $\mathfrak{m}.$ Suppose
that
\[
\rho\pmod{J}=\rho_1'\oplus\cdots \oplus \rho_n'
\]
 with
$\rho_i':G_F\longrightarrow GL_{d_i}(R/J)$ lifting $\rb_i.$ The obstruction to lifting
$\rho_i'$ to a representation $G_F\longrightarrow GL_{d_i}(R)$ is a cohomology class
 \[
 c_i\in H^2(G_F, \Hom(\rb_i\otimes J,\rb_i\otimes
J))=H^2(G_F,\Hom(\rb_i, \rb_i))\otimes J.
\]
 Since $\rho\pmod{J}$
lifts to $R,$ $c_1+\ldots+c_n$ vanishes in $H^2(G_F,
\Hom(\rb, \rb))\otimes J.$ Hence $c_1,\ldots, c_n$ are trivial
by the first part of Lemma~\ref{H^i}.

We can therefore find lift each $\rho_i'$ to
$\widetilde\rho_i:G_F\longrightarrow GL_{d_i}(R)$.
If we set
$\widetilde\rho:=\widetilde\rho_1\oplus\cdots\oplus\widetilde\rho_n$, then
$\rho=(I+\xi)\widetilde\rho$ with $\xi\in H^1(G_F,\Hom(\rb\otimes
J,\rb\otimes J))$. By the first part of  Lemma~\ref{H^i}, we see that $\xi=\xi_1+\ldots +\xi_n$ with
$\xi_i\in H^1(G_F, \Hom(\rb_i\otimes J,\rb_i\otimes J))$. The
required decomposition for $\rho$ follows. The uniqueness part follows from the second part of Lemma \ref{H^i}. \end{proof}

\subsection{Tamely ramified representations} \label{tamely ramified}

We now consider the problem of constructing a well-behaved deformation condition when the residual representation is tamely ramified. Throughout this subsection, $F$ is a fixed finite extension of $\Q_p$ with residue field of order $q$. We denote by
$F^\text{nr}$ and $F^\text{tr}$ the maximal unramified and the maximal tamely ramified extensions of $F$, and  fix
\begin{itemize}
 \item a topological generator $\tau$ of
$\text{Gal}(F^\text{tr}/F^\text{nr}),$
\item  a lift $\sigma$ of Frobenius to
$\Gal(F^\text{tr}/F).$
\end{itemize}

The letter $T$ denotes a fixed indeterminate. For  a tamely ramified representation
$\rho:G_F\longrightarrow GL_n(R),$ we shall view the underlying module $V(\rho)$ as an
$R[T]$-module where $T$ acts via $\tau$. (We shall freely identify tamely ramified
representations with representations of $\text{Gal}(F^\text{tr}/F)$.)  Note that the action of $\sigma$ provides
added structure.

To describe this further, we first fix some notation:
\begin{itemize}
\item $\phi_{q} :R[T]\longrightarrow R[T]$ is the injective homomorphism which sends
$T$ to $T^q$ (and is the identity on $R$).

\item If $M$ is an
$R[T]$-module,  then  $\phi_{q}^*M$ is  the $R[T]$-module
with underlying set $M$
and action twisted by $\phi_q$ i.e.
$(f(T),m)\longrightarrow f(T^q)m$ for all $f(T)\in R[T]$.
\end{itemize}
 Then, with notation as before,
specifying the action of $\sigma$ on $V(\rho)$ is equivalent to specifying an isomorphism
$V(\rho)\longrightarrow \phi_q^*V(\rho)$ of $R[T]$-modules. Conversely, these determine
the representation completely.

Now let $\rb: G_F\longrightarrow GL_n(k)$ be a tamely ramified representation.
and let $(a_{ij})$ be
the (upper triangular) Jordan normal form of $\rb(\tau)$ (so $a_{ij}=0$ if $i<j$ or
$i>j+1$, and $a_{i,i+1}$ is $0$ or $1$). We define the $n\times n$ matrix $J(\rb)$ by
\[
J(\rb):=(\widehat{a_{ij}})\quad \text{where $\widehat{a_{ij}}$ is the Teichm\"{u}ller lift of $a_{ij}$.}
\]
 Finally, let
$\cD_{J(\rb)}$ be the full subcategory of $\rep_n(G_F;k)$ consisting of objects
$(A,\rho)$ with $\rho:G_F\longrightarrow GL_n(A)$ tamely ramified and
$\rho(\tau)\sim J(\rb).$ We then have the following:

\begin{proposition} $\cD_{J(\rb)}$ determines a well-behaved
deformation condition for $\rb$.\label{tameprop}\end{proposition}

We'd like to study deformations $(R,\rho)$ in $\cD_{J(\rb)}$ using the linear algebra
data `$R[T]$-module with added structure', and for that we need a convenient
description of $J(\rb)$ in terms of $R[T]$-modules.

Recall that $k$ is a finite of characteristic $\ell\neq p$. We denote by $k_{(q)}$ the orbits of the action $\alpha\longrightarrow \alpha^q$ on the set of
elements in $k^\times$ which have order prime to $q$. For $\alpha\in k^\times$ with order
prime to $q$ we define the polynomial
\[
P_\alpha(T):=
\big(T-\widehat\alpha\big)\big(T-\widehat\alpha^q\big)\cdots
\big(T-\widehat\alpha^{q^d}\big)
\]
where $d$ is the smallest non-negative integer with $\alpha^{q^{d+1}}=\alpha$.
As usual, $\widehat{\alpha}\in W$ denotes the Teichm\"{u}ller lift of
 $\alpha\in k$. Equivalently, $P_\alpha$ is the polynomial whose roots are the
 Teichm\"{u}ller lifts of elements in the orbit of $\alpha$.
Finally, if $\mathbf{x}\in k_{(q)}$ is the orbit of $\alpha$ then
$P_\mathbf{x}:=P_\alpha$.

\begin{definition} \mbox{}
\begin{enumerate}
\item A type function $\T$ is a map  $\T:k_{(q)}\times \mathbb{N}\longrightarrow \Z$
such that
\begin{itemize}
\item $\T(\mathbf{x},m)\geq \T(\mathbf{x},m+1)$ for all $\mathbf{x}\in k_{(q)}$,
$m\in \mathbb{N}$, and
\item $\T(\mathbf{x},m)=0$ for almost all $\mathbf{x},$ $m$.
\end{itemize}
\item
 Let $R$ be a CNL $W$-algebra, and let $\T$ be a type function. The
 \emph{standard  $R[T]$ module of type $\T$}, denoted by $J(R,\T)$, is
\[
\bigoplus_{\mathbf{x}\in k_{(q)}}\left(
\frac{R[T]}{\left(P_{\mathbf{x}}^{\T(\mathbf{x},1)}\right)}\oplus
\frac{R[T]}{\left(P_{\mathbf{x}}^{\T(\mathbf{x},2)}\right)}
 \oplus\cdots\right).
\]
An $R[T]$ module $M$ is said to be of type $\T$  if $M$ is isomorphic to $J(R,\T)$.
A tamely ramified representation $\rho:G_F\longrightarrow GL_n(R)$ is said to be of
type $\T$ if
the underlying module $V(\rho)$  is of type $\T$.
\end{enumerate}
\end{definition}

We make the following observation. Let $\rb: G_F\longrightarrow GL_n(k)$ be a tamely
ramified representation. Because $\sigma\tau\sigma^{-1}=\tau^q$, the uniqueness of
Jordan normal form implies that $V(\rb)$ is a $k[T]$-module of type $\T$ for some type
function $\T$. Fix one such type function $\T$. Then $(A,\rho)$ is in $\cD_{J(\rb)}$ if
and only if $\rho$ is of type $\T$.

We now establish some results that will be needed in the proof of our key proposition
\ref{tameprop}.

\begin{lemma} Let $\alpha,\beta\in k^\times$ have orders prime to
$q$ and let $f:R\longrightarrow S$ be a surjective homomorphism of Artinian CNL algebras. Given
$m,n\geq 1$ and $\phi\in\Hom_{S[T]}\left(S[T]/(P_\alpha^m),S[T]/(P_\beta^n)\right),$ there exists
$\widetilde\phi\in\Hom_{R[T]}\left(R[T]/(P_\alpha^m),R[T]/(P_\beta^n)\right)$ such that the
diagram
\[
\begin{matrix}
R[T]/(P_\alpha^m)&\stackrel{\widetilde\phi}{\longrightarrow}&R[T]/(P_\beta^n)\\
\downarrow&&\downarrow\\
S[T]/(P_\alpha^m)&\stackrel{\phi}{\longrightarrow}&S[T]/(P_\beta^n)\end{matrix}
\]
 commutes.
\label{lifting1}\end{lemma}

\begin{proof}  The lemma holds trivially if $\alpha\neq \beta^{q^j}$ for any $j\geq 0$
because
\[
\Hom_{R[T]}\left(
R[T]/(P_\alpha^m),R[T]/(P_\beta^n)\right)=(0)
\]
 in this case.

Suppose now that $\alpha=\beta.$ To give an $S[T]$-module homomorphism
$\phi:S[T]/(P_\alpha^m)\longrightarrow S[T]/(P_\alpha^n)$ is equivalent to finding a $g(T)\in
S[T]$ such that $P_\alpha^mg(T)\in (P_\alpha^n)$ (and  $\phi(1)= g(T)\pmod{P_\alpha^n} $). If
$m\geq n,$ take $\widetilde g(T)\in R[T]$ to be a lift of $g(T),$ and define
\[
\widetilde\phi:R[T]/(P_\alpha^m)\longrightarrow
R[T]/(P_\alpha^n)
\]
 by setting $\widetilde\phi(1)=\widetilde g(T)\pmod{P_\alpha^n}.$  If $m<n,$
we have $g(T)=P_\alpha^{n-m} h(T)$ for some $h(T)\in S[T].$ In this case, define
\[
\widetilde\phi(1):=P_\alpha^{n-m} \widetilde h(T)\pmod{P_\alpha^n}
\]
where $\widetilde h(T)\in R[T]$ is a lift of $h(T).$
\end{proof}

\begin{proposition} Let $R$ be an  Artinian CNL algebra, and let
$I$ be an ideal of $R.$ If $M,N$ are  $R[T]$-modules of type $\T_M,\T_N$ respectively, then any
$R[T]$-module homomorphism $M/IM\longrightarrow N/IN$ lifts to a homomorphism $M\longrightarrow
N.$\label{lifting2}\end{proposition}

\begin{proof} Fix isomorphisms
\[
\theta_M:M\longrightarrow \bigoplus
\frac{R[T]}{P_\alpha^{\T_M(\alpha,i)}},\ \theta_N:N\longrightarrow \bigoplus
\frac{R[T]}{P_\alpha^{\T_N(\alpha, i)}},
\]
 and let $\overline\theta_M,\overline\theta_N$ be their
reductions modulo $I.$ Given a homomorphism of $R[T]$-modules $\overline\phi: M/IM\longrightarrow
N/IN,$ we can apply Lemma~\ref{lifting1} to find a lift
\[
\psi:\bigoplus
\frac{R[T]}{P_\alpha^{\T_M(\alpha,i)}}\longrightarrow\bigoplus
\frac{R[T]}{P_\alpha^{\T_N(\alpha,i)}}
\]
  of
$\bar\theta_N{\bar\phi}{\bar\theta_M}^{-1}.$ If we now take
$\phi:M\longrightarrow
N$ to be $\theta_N^{-1}\psi \theta_M,$ then $\phi\pmod{I}=\overline\phi.$ \end{proof}

\begin{proposition} Let $R$ be a CNL $W$-algebra.  Let
$\phi_q :R[T]\longrightarrow R[T]$ be the injective homomorphism sending $T$ to $T^q.$ Then
$\phi_q$ induces an isomorphism
\[
\frac{R[T]}{P_\alpha^n}\longrightarrow
\frac{R[T]}{P_\alpha^n}
\]
  of $R$ algebras for any $\alpha\in k^\times$ of order coprime to $q$, $n\geq 1.$

 Consequently, if  $M$ is an
$R[T]$-module of type $\T$ then $\phi_q^*M$ is also of type $\T.$

\label{qfrob}\end{proposition}

\begin{proof} First suppose that $R$ is Artinian.
Suppose we have a polynomial $f(T)\in R[T]$ with
\[
f(T^q)=P_\alpha(T)^n g(T)
\]
for some $g(T)\in R[T].$ Then
\[
f(\widehat\alpha)=f(\widehat\alpha^q)=\cdots=0.
\]
 Since
$\widehat\alpha^{q^i}-\widehat\alpha^{q^j}$ is a unit if $0\leq i<j<d_\alpha,$ we have
$f(T)=P_\alpha h(T)$ for some $h(T)\in R[T].$ Now
\begin{align*}
P_\alpha(T^q)&=\prod_{i=0}^{d_\alpha-1}\left(T^q-\widehat\alpha^{{q^i}q}\right)\\
&= P_\alpha(T)\prod_{\zeta^q=1 \atop \zeta\neq 1}
\prod_{i=0}^{d_\alpha-1}\left(T-\zeta\widehat\alpha^{{q^i}}\right)\\
&=P_\alpha(T)\prod_{\zeta^q=1\atop\zeta\neq 1}P_\alpha(\zeta T),
\end{align*}
and therefore
\[
P_\alpha(T)^{n-1} g(T)=h(T^q)\prod_{\zeta^q=1\atop\zeta\neq 1}P_\alpha(\zeta
T).
\]
 Since $\widehat\alpha^{q^j}-\zeta\widehat\alpha^{{q^i}}$ are units, we have
\[
h(\widehat\alpha)=h(\widehat\alpha^{q^2})=\cdots=0.
\]
We can now conclude (by induction) that $\phi_q$ induces an injection
\[
\frac{R[T]}{P_\alpha^n}\longrightarrow
\frac{R[T]}{P_\alpha^n},
\]
  and therefore induces an isomorphism.

The non-Artinian case follows on taking inverse limits.
\end{proof}

\begin{proof}[Proof of Proposition \ref{tameprop}] To show that $\cD_{J(\rb)}$ determines a
deformation condition, we
 need only verify condition (DC2) as  (DC1) is obvious.
Fix a type function $\T$ so that $\rb$ is of type $\T$. Let
\[
\begin{CD}
A\times_CB @>{\pi_B}>> B\\
@V{\pi_A}VV @V{\beta}VV\\
A@>{\alpha}>>C
\end{CD}
\]
 be a cartesian diagram of Artinian local $W$-algebras with
$\beta$ small, and let $(A\times_CB,\rho)$ be an object in $\rep_N(G_F;k)$ with $\pi_A\rho,
\pi_B\rho_B$ of type $\T$. We need to show that $\rho$ is of type $\T.$

Let $(t)$ be the kernel of $\beta.$ Then $\pi_A$ is small with kernel generated by $(0,t).$ We
may suppose that $\pi_A\rho|_{I_F}=\rho_{\T}$, and so
$\rho|_{I_F}=(I+(0,t\xi))\rho_{\T}$ with $\xi$
a $1$-cocycle representing an element of $H^1(I_F,\Ad\rb).$ We need to show that $\xi$
is trivial.

Now $\pi_B\rho|_{I_F}=(I+t\xi)\rho_{\T}$, and also
$M\pi_B\rho M^{-1}|_{I_F}=\rho_{\T}$ for some
$M\in GL_N(B)$. Going down to $C=B/(t),$ we have that $\beta(M)$ commutes with
$\rho_{\T}$. Using
Proposition \ref{lifting2}, we can find $M'\in GL_N(B)$ such that
$M'\rho_{\T}{M'}^{-1}=\rho_{\T}$ and
$M\equiv M'\pmod{t}.$ Thus $\rho_{\T}=(I+tX)\rho(I-tX)|_{I_F}$ for some $N\times N$
matrix over
$k$, and hence $\xi$ is trivial.

Let $R\longrightarrow S$ be a surjective morphism of Artinian local $W$-algebras, and let
$\rho_S:G_F\longrightarrow GL_N(S)$ be a deformation of type $\T$. Conjugating $\rho_S$ by a
matrix congruent to the identity modulo the maximal ideal of $S$, we may suppose that $V(\rho_S)$
is $J(S,\T)$. The action of $\sigma$ specifies a morphism
\[
\theta_S:J(S,\T)
\longrightarrow \phi_q^*J(S,\T)
\]
 of $S[T]$-modules  which can then be lifted, by
Proposition \ref{lifting2}, to
\[
\theta_R:J(R,\T)
\longrightarrow \phi_q^*J(R,\T).
\]
 Hence $\cD_{\T}$ is smooth.

The deformations of $\rb$ to $k[\epsilon]/\epsilon^2$ are uniquely determined by
$H^1(G_F,\Ad\rb).$ For $\xi\in H^1(G_F,\Ad\rb),$ the lift
$(I+\epsilon\xi)\rb$ is of type $\T$ if and only if the restriction of $\xi$ to inertia
is trivial. Hence the tangent space for $\cD_{J(\rb)}$ is
$H^1\left(G_F/I_F,(\textrm{ad}\rb)^{I_F}\right).$ Hence $\cD_{J(\rb)}$ is a well
behaved deformation condition.
\end{proof}

\subsection{Deformations for tensor products} \label{tensor prod}

We now consider the problem of constructing well-behaved deformations using tensor products. As in the preceding sections, $F$ is a finite extension of $\Q_p$ and $k$ is finite field of
characteristic $\ell, \ell\neq p.$ Fix a residual representation
$\overline\theta:G_F\longrightarrow GL_n(k)$ such that
\begin{itemize}
\item $\overline\theta$ is  absolutely irreducible,
\item $\ell\nmid n,$ and
\item $\overline\theta$ is not equivalent to its Tate twist $\overline\theta(1).$
 \end{itemize}
We set $s$ to be the smallest positive integer such that
$\overline\theta(s)\sim\overline\theta.$ (So $s\geq 2$ by our assumption.) We then have the following.

\begin{theorem} \label{tensortheorem}
Suppose that $1\leq m\leq s-2,$ and let $\rb:G_F\longrightarrow GL_{mn}(k)$ be
a representation such that $\rb^\textrm{ss}\cong
\overline\theta(a_1)\oplus\cdots\oplus\overline\theta(a_m).$ There is then a deformation
condition $\mathcal{E}$ for $\rb$ with the following properties:
\begin{itemize}
\item If $(A,\rho_A)\in\mathcal{E},$ then $\det\rho_A$ restricted to the inertia subgroup
of $G_F$ is the Teichm\"{u}ller lift of $\det\rb ;$ \item $\mathcal{E}$ is a
smooth deformation condition;
\item The dimension of the tangent space for $\mathcal{E}$ is equal
to $\dim H^0(G_F,\Ad\rb).$
\end{itemize}
\end{theorem}

We make the following definition for convenience: A
representation $r:G_F\longrightarrow GL_d(k)$ is said to be \emph{$s$-small} if
\[
r^\textrm{ss}\cong k(i_1)\oplus \cdots\oplus k(i_d)
\]
with $0\leq i_1,\ldots , i_m\leq s-2.$

We shall make use of the natural isomorphism between $\Hom(V,W)$ and $ V^\vee\otimes W$ for
$k$-vector spaces $V,W$ in what follows without any further qualification. Also, the identity map
on $U$ naturally identifies $\Hom(V,W)$ as a subspace of $\Hom(V\otimes U,W\otimes U).$ If
$\ell\nmid \dim U,$ then $\Hom(V\otimes U,W\otimes U)$ is naturally identified with
$\Hom(V,W)\oplus \Hom(V,W)\otimes \Ad^0U$ where $\Ad^0U$ is the vector space of trace zero
endomorphisms of $U.$

\begin{lemma}\label{tensorlemma}
\mbox{}
\begin{enumerate}
\item[(a)] If $|j|\leq s-2$ then $H^i\left(G_F,\Ad^0\overline\theta(j)\right)=(0)$ for all $i\geq
0.$ \item[(b)] If $0\leq a,b\leq s-2$ then the decomposition described above induces natural
isomorphisms
$$H^i\left(G_F,\Hom\left(\overline\theta(a),\overline\theta(b)\right)\right)\cong
H^i\left(G_F,k(b-a)\right)$$ for all $i\geq 0.$ \item[(c)] If $\rb_1,\rb_2$
are two $s$-small representations then the natural inclusion
$\Hom\left(\rb_1,\rb_2\right)\hookrightarrow
\Hom\left(\rb_1\otimes\overline\theta,\rb_2\otimes\overline\theta \right)$
induces isomorphisms
\[
H^i\left(G_F,\Hom\left(\rb_1\otimes\overline\theta,\rb_2\otimes\overline\theta
\right)\right)\cong H^i\left(G_F,\Hom\left(\rb_1,\rb_2 \right)\right)
\]
 for all $i\geq 0.$
\end{enumerate}
\end{lemma}

\begin{proof} For part ({\it a}), one checks that the statement holds for $|j|\leq s-1$ when $i=0.$ The
full result then follows after an application of local Tate duality  and the Euler characteristic
formula. Part ({\it b}) of the lemma is then immediate from ({\it a}).

For part ({\it c}), we have
\[
\Hom\left(\rb_1\otimes\overline\theta,\rb_2\otimes\overline\theta
\right) \cong \Hom(\rb_1,\rb_2)\oplus
\Hom(\rb_1,\rb_2)\otimes\Ad^0\overline\theta,
\]
 and $H^i\left(G_F,
\Hom(\rb_1,\rb_2)\otimes\Ad^0\overline\theta\right)$ is trivial by part {\it
a.}
\end{proof}

Let $\theta:G_F\longrightarrow GL_n(W)$ be the unique (up to equivalence) lifting of
$\overline\theta$ with determinant the Teichm\"{u}ller lift of $\det\overline\theta.$
(The existence and uniqueness of such a representation is an immediate consequence of the above
lemma.) Fix also an $s$-small representation $\rb_0:G_F\longrightarrow GL_m(k)$ and a
deformation condition $\cD$ for $\rb_0.$

Define $\cD\otimes \theta$ to be the full subcategory of $\rep_{mn}(G_F)$ whose objects
are pairs
 $(A,\rho_A)$ with $\rho_A\sim \rho_0\!\otimes\theta$ for some
 $(A,\rho_0)\in \cD.$

 \begin{proposition} With notation as above, $\cD\otimes\theta$ is a deformation condition
 for $\rb_0\otimes\overline\theta.$ The tangent space for $\cD\otimes\theta$ is
 naturally identified with $\cD.$
\label{tensorprop}
 \end{proposition}

 \begin{proof} We first show that $\cD\otimes\theta$ is a deformation condition, and for that we
 need only verify that a lifting $\rho:G_F\longrightarrow A\times_BC$ is in  $\cD\otimes\theta$
 if the projections of $\rho$ to $A$ and $C$ are in $\cD\otimes\theta.$

\bigskip\noindent
 {\it Claim 1:} If $\rho: G_F\longrightarrow GL_{mn}(A)$ is a lifting of
 $\rb_0\otimes\overline\theta,$ then $\rho$ is strictly equivalent to
 $\rho_0\otimes\theta$ for some lifting $\rho_0:G_F\longrightarrow GL_m(A)$ of $\rb_0.$

 \medskip\noindent {\it Proof of claim:} We use induction on length for $A$ Artinian. Let $J$ be
 an ideal of $A$ killed by the maximal ideal $\m$ of $A.$ Then $\rho\ \textrm{mod}\ J$
is strictly equivalent to $\rho_1\otimes\theta$ for some lift to $A/J$ of $\rb_0.$ The
obstruction to lifting $\rho_1$ to $GL_m(A)$ lies in $H^2(G_F,\Ad\rb_0)\otimes J,$ and
the obstruction vanishes by Lemma~\ref{tensorlemma}, part {\it c.} We can therefore find a
lifting $\rho_0':G_F\longrightarrow GL_m(A)$ of $\rb_0$ such that
 $\rho\ \textrm{mod}\ J=\rho_0'\otimes \theta\
\textrm{mod}\ J.$ It follows that $\rho=\rho_0'\otimes\theta\left(1+\xi\right)$ for some $\xi\in
H^1(G_F,\Ad\rb_0\otimes\overline\theta)\otimes J,$ and the claim follows from
Lemma~\ref{tensorlemma}, part {\it c.}

\bigskip\noindent
{\it Claim 2:} If $\rho_1,\rho_2:G_F\longrightarrow GL_m(A)$ are two liftings of
$\rb_0$ and $\rho_1\otimes\theta\sim_\textrm{s}\rho_2\otimes\theta,$ then
$\rho_1\sim_\textrm{s}\rho_2.$

\medskip\noindent
{\it Proof of claim:} With $A,J$ as in the proof of claim 1 and using induction on length, one
deduces that assuming $\rho_1\ \textrm{mod}\ J=\rho_2\ \textrm{mod}\ J,$ we have
$\rho_1\otimes\theta=\rho_2\otimes\theta (1+\xi)$ with $\xi\in
H^1(G_F,\Ad\rb_0\otimes\overline\theta )\otimes J.$ Lemma~\ref{tensorlemma} again
completes the proof.

\bigskip
Now let $(A\times_BC,\rho)$ be a lifting of $\rb_0\otimes\overline\theta.$ We may
assume by claim 1 that $\rho=\rho_0\otimes\theta$ for  $\rho_0$ a  lifting of $\rb_0.$
If the projections of $\rho$ to $A$ and $C$ are in $\cD\otimes\theta,$ then claim 2
implies that the projections of $\rho_0$ to $A$ and $C$ are in $\cD.$ Hence
$(A\times_BC,\rho_0)\in\cD,$ thus proving the theorem.

The statement about tangent spaces is immediate from Lemma~\ref{tensorlemma}.\end{proof}

\begin{proof}[Proof of Theorem \ref{tensortheorem}] Twisting $\rb$ by a power of the cyclotomic character, we may assume that
$0\leq a_1,\ldots , a_m \leq s-2.$ It is then easy to see, using Lemma~\ref{tensorlemma}, that
$\rb\sim\rb_0\otimes \overline\theta$ where $\rb_0$ is a $s$-small
representation with $\rb_0^\textrm{ss}\cong k(a_1)\oplus\cdots\oplus k(a_m).$ Now let
$\mathcal{E}_0$ be the deformation condition for the tamely ramified representation
$\rb_0$ constructed in subsection~\ref{tamely ramified}, and take $\mathcal{E}$ to be
the deformation condition $\mathcal{E}_0\otimes\theta.$ All claims then follow from
Proposition~\ref{tensorprop} and properties of $\mathcal{E}_0.$ \end{proof}

\subsection{Induced representations} \label{induced repn}

Let $F\subsetneqq L$ be fixed   finite extensions of $\Q_p$. Set $n=[L:F]$. We assume we are given
a representation
$\rb:G_F\longrightarrow GL_{mn}(k) $ which is  induced from
$\overline\theta:G_L\longrightarrow GL_m(k).$
Let's fix a coset decomposition
\[
G_F=g_1G_L\sqcup\cdots\sqcup g_{n}G_L
\]
 with $g_1=e.$ Then $V(\rb)$
has a $G_L$ invariant vector subspace $M$ such that:
\begin{itemize}
\item $V(\overline\theta)\cong M$ as $G_L$-modules, and
\item $V(\rb)=g_1M\oplus\cdots\oplus g_nM.$
\end{itemize}
The subspace $N:=g_2M+\cdots g_{n}M$ is $G_L$ invariant and $V=M\oplus N$ as $G_L$-modules. Let
$\overline\vartheta:G_L\longrightarrow GL_{(n-1)m}(k)$ be a representation given by (some fixed
choice of basis of) $N.$ Assume that:
\begin{itemize}
\item $\rb|_{G_L}=\overline\theta\oplus\overline\vartheta$, and \item
$\Hom_{G_L}\left(M,N(r)\right)=(0)$ for all $r\in\Z.$\end{itemize} Under these assumptions,  we
have canonical isomorphisms
\[
H^i(G_F,\Ad\rb)\cong H^i(G_L,\Ad\overline\theta)
\]
 by Shapiro's lemma.
Furthermore, Proposition \ref{lifting2} shows that any lift $\rho:G_F\longrightarrow GL_{mn}(R)$ of
$\rb$ restricted to $G_L$ is strictly equivalent to $\theta\oplus\vartheta$ where
$\theta, \vartheta$ are lifts of $\overline\theta$ and $\overline\vartheta.$

\begin{lemma}\label{ind1}
 Let $A$ be an Artinian CNL $W$-algebra, and
let $\rho: G_F\longrightarrow GL_{mn}(A)$ be a lift of $\rb.$ If
$$\rho|_{G_L}=\theta\oplus\vartheta$$ with $\theta,\vartheta$
lifts of $\overline\theta, \overline\vartheta,$ then $\rho$ is equivalent to
$\textrm{Ind}\,\theta.$\end{lemma}

\begin{proof} We fix a basis for $V(\rb)$ as follows:
View $V(\overline\theta)$ as a subspace of $V(\rb)$ via
$V(\rb)=V(\overline\theta)\oplus V(\overline\vartheta),$ and take the basis
$\{g_i\overline\e_j\,|\, 1\leq i\leq n,\,1\leq j\leq m\}$ with
$\{\overline\e_1,\ldots,\overline\e_m\}$ a basis of $V(\overline\theta).$ Now
$V(\rho)=V(\theta)\oplus V(\vartheta)$ as $A[G_L]$-modules, and so we can pick a basis
$\{\e_1,\ldots,\e_m\}$ of $V(\theta)$ such that $\e_i$ is a lift of $\overline\e_i.$ It is now
clear that
\[
V(\rho)=g_1V(\theta)+\cdots+g_nV(\theta)+\mathfrak{m}_AV(\rho),
\]
and therefore, by Nakayama's lemma, one sees that
\[
V(\rho)=g_1V(\theta)\oplus\cdots\oplus g_nV(\theta).
\]
This completes the proof (using, for instance, Proposition~10.5 of \cite{CR}).
\end{proof}

Now let $\cF$ be a deformation condition for $\theta$,
and denote by $\textrm{Ind}\cF$ the full subcategory of $\rep_{mn}(G_F;k)$
whose  objects
are $(A,\rho)\in\rep_{mn}(G_F;k)$ with $V(\rho)\cong \textrm{Ind}V(\theta)$ for some
$(A,\theta)\in\cF.$

\begin{proposition}  $\textrm{Ind}\cF$ is a deformation
condition for $\rb.$ If $\cF$ is well-behaved then so is
$\textrm{Ind}\cF.$  \label{ind2}\end{proposition}

\begin{proof} To show that $\textrm{Ind}\cF$ is a
deformation condition, we need only check  (DC2). Suppose given $\alpha:A\longrightarrow
C,\beta:B\longrightarrow C,$ with $\beta$ small, and a lift
\[
\rho:G_F\longrightarrow GL_{mn}\left(A\times_CB\right)
\]
of $\rb$ with $(A,\alpha\rho),(B,\beta\rho)$ in
 $\textrm{Ind}\cF\rep_{mn}.$
Conjugating by an element of $GL_{mn}(A\times_CB),$ we can take $\rho$ to be a lift of
$\rb,$ and that
$\rho|_{G_L}=\theta\oplus\vartheta$
where $\theta,\vartheta$ are lifts of $\overline\theta$ and $\overline\vartheta.$  Since
$\rho\sim\textrm{Ind}\,\theta$ by Lemma~\ref{ind1}, we need to verify that $(A\times_CB,\theta)$
is in $\cF\rep_m.$

Let $(A,\theta')$ be an object of $\cF\rep_m$ with $\textrm{Ind}\theta'\sim \alpha\rho.$
By Proposition~\ref{H2}, the composite
\[
V(\theta')\hookrightarrow V(\alpha\rho)\cong
V(\alpha\theta)\oplus V(\alpha\vartheta)\longrightarrow V(\alpha\theta)
\]
 is an isomorphism of
$A[G_L]$-modules. Hence $(A,\alpha\theta)$ is an object of $\cF\rep_m$. Similarly,
$(B,\beta\theta)$ is an object of $\cF\rep_m$, and hence $(A\times_CB,\theta)$ is in
$\cF\rep_m$.

Clearly, $\textrm{Ind}\cF$ is smooth if $\cF$ is, and the tangent space for
$\textrm{Ind}\cF$ is the image of $T\cD$ under the Shapiro isomorphism. The
(uni)versal deformation ring for  is a power series ring over $W$, the restriction of the
determinant of the (uni)versal $\textrm{Ind}\cF$ deformation is the Teichm\"{u}ller lift
of $\det\rb$. The second statement of the proposition now follows.\end{proof}

\subsection{Proof of Theorem \ref{main local}}
\label{main local proof}

Recall we are assuming  that our representation $\rb:G_F\longrightarrow GL_N(k)$  has all irreducible
components occurring in the semi-simplification of $\rb$ absolutely irreducible, and that  $[F(\zeta_\ell):F]\geq 3N$ for
$p\leq N$. Our task is to construct a well-behaved deformation condition for $\rb$.
Let's fix absolutely irreducible continuous representations
\[
\overline\theta_i:G_F\longrightarrow GL_{n_i}(k),\ i=1,\ldots ,n
\]
 such that:
\begin{itemize}
\item if $i\neq j,$ then $\overline\theta_i$ and $\overline\theta_j(r)$ are not equivalent for
any $r\in\Z ;$ \item $\rb^{ss}$ is a direct sum of $\overline\theta_i,\ i=1,\ldots, n,$
and Tate twists of $\overline\theta_i$'s.
\end{itemize}

\begin{lemma} Let $V$ be the underlying $k[G_F]$-module for
$\rb.$ Then $V$ has a submodule isomorphic to $V(\overline\theta_i)$ for each $i.$ If
$V_i$ denotes the maximal submodule of $V$ whose composition series consists only of
$\overline\theta_i$ and Tate twists of $\overline\theta_i,$ then
$V=V_1\oplus\cdots\oplus V_n.$
Furthermore, for any $r\in\Z, i\neq j,$ we have
\[
\Hom_{G_F}\left(V_i,V_j(r)\right)=(0).
\]
\end{lemma}

\begin{proof} We may suppose that $V$ has a submodule $U$
isomorphic to $\overline\theta_1.$ Using induction, we get an exact sequence of $k[G_F]$ modules
\[
0\longrightarrow U\longrightarrow V\longrightarrow M_1\oplus
\cdots \oplus M_n\longrightarrow 0
\]
 where each $M_i$ composition series consisting only of
$\overline\theta_i$ and Tate twists of $\overline\theta_i.$ Thus $V$ corresponds to an element of
\[
H^1\left(G_F, \Hom(M_1\oplus
\cdots \oplus M_n, U)\right).
\]
 By Tate local duality, $H^1(G_F,\Hom(M_i, U))$ is trivial
if $i\neq 1,$ and the proposition follows.\end{proof}

By Theorem~\ref{directsum} and the above lemma, we can assume that the semi-simplification of
$\rb$ is a direct sum of Tate twists of a single absolutely irreducible representation
$\overline\theta: G_F\longrightarrow GL_n(k).$ If $\overline\theta$ is tamely ramified, we
proceed as in subsection \ref{tamely ramified}, Proposition~\ref{tameprop}.

Now assume that $\overline\theta$ is wildly ramified. We shall deal with the case when $p\leq N$
first. Let $s$ be the smallest positive integer such that $\theta\sim \theta(s),$ and let $m$ be
the number irreducible components of $\rb^\textrm{ss}$ isomorphic to some Tate twist of
$\overline\theta.$ The inequalities $ns\geq 3N$ (obtained by comparing determinants
of $\theta$ and $\theta(s)$) and $nm\leq N$ imply that $1\leq m\leq s-2$. The
existence of a well-behaved deformation condition then follows from Theorem~\ref{tensortheorem}.

Finally, assume from here on that $\overline\theta$ is wildly ramified and $p>N.$ Let
$\rb^{ss}\cong \overline\theta(i_1)\oplus\cdots\oplus\overline\theta(i_m),$  and denote
by $F(\rb)$ the extension of $F$ through which $\rb$ factorises. Since $n<p$
the $p$-part of the determinant of $\overline\theta$ can be made trivial after twisting by a
character $G_F\longrightarrow k^\times$. A consideration of ramification subgroups shows that we
can find an abelian normal,wildly ramified, $p$-subgroup $Z\lhd\Gal(F(\rb)/F)$. The assumption we just
made on the determinant shows that $\overline\theta|_Z$ is not central.

We now give a characterisation of $\rb$ as an induced module.  The representation
$\rb$ when restricted to $Z$ splits as a direct sum of characters. Clearly, if
$\overline\theta|_Z\sim \chi_1\oplus\cdots\oplus\chi_d,$ then
$\rb|_Z\sim\left(\chi_1\oplus\cdots\oplus\chi_d\right)^{mn/d}.$ We fix one such
character $\chi $ and set
\[
V[\chi]:=\left\{v\in
V(\rb)\,|\,\rb(z)(v)=\chi(z)v\quad\text{for all $z\in Z$}\right\}.
\]
 If
$g\in\Gal\left(F(\rb)/F\right),$ then the character ${}^g\!\chi$ defined by
\[
{}^g\!\chi(z):=\chi(gzg^{-1})
\]
 is also a constituent character
of $\overline\theta|_Z,$ and we have $V[{}^g\!\chi]=gV[\chi].$ Thus
$\Gal\left(F(\rb)/F\right)$ acts transitively on the distinct constituent characters of
$\overline\theta|_Z$  and there are at least two distinct constituent characters. Let $L$ be the
finite extension of $F$ inside $F(\rb)$ cut out by the stabiliser of $\chi,$ and fix a
coset decomposition
\[
G_F=g_1G_L\sqcup\cdots\sqcup g_{n}G_L
\]
 with $g_1=e.$
Then
$V= g_1V[\chi]\oplus\cdots\oplus g_nV[\chi],$ and so $V$ is induced from
the $G_L$-module $V[\chi].$ Since $\chi$ is a wildly ramified character,
\[
\Hom_Z\left( V[{}^g\!\chi],
V[{}^{g'}\!\chi]\right)=(0)
\]
 if $gG_L\neq g'G_L,$ and so for any $r\in\Z,$ we have
\[
\Hom_{G_L}\left(V[\chi], (g_2V[\chi]\oplus\cdots\oplus
g_nV[\chi])(r)\right)=(0).
\]

Finally, inductively on $N,$ one can find a well-behaved deformation condition for the
representation of $G_L$ arising from $V[\chi].$ Using Theorem~\ref{ind2}, the induced deformation
condition is a well-behaved deformation condition for $\rb.$

\subsection{Deformations at special unramified primes\label{ramakrishna}}

We  conclude this section with a look at a special class of smooth local deformations which are of great significance
in reducing dimensions of (global) dual Selmer groups.
So let $F$ be a finite extension of $\Q_p$ and let $\rb:G_F\longrightarrow GL_n(k)$ be the diagonal representation
\[
\rb= \begin{pmatrix}
\bar\omega^{n-1}&&&\\
&\bar\omega^{n-2}&& \\
&&\ddots&\\
&&& 1\end{pmatrix},
\]
We assume that the order of the mod $\ell$ cyclotomic character
$\bar\omega$ is greater than $n.$ Fix a
 basis $\{\e_1,\e_2,\ldots, \e_n\}$
with $\rb$ acting on $\e_i$ by the character $\bar\omega^{n-i}.$

We write $B_n$ for the standard Borel subgroup of $GL_n$ consisting of upper triangular matrices and set
\[
b\rb:=\bigoplus_{1\leq i\leq j\leq n}\Hom(k\e_j,k\e_i).
\]
Note that
\[
\Ad\rb\cong\bigoplus_{1\leq i,j\leq n}\Hom(k\e_j,k\e_i)
\cong\bigoplus_{1\leq i,j\leq n}k(i-j)
\]
 as $G_F$-modules. This identification will be fixed.

\begin{proposition} Let $\rho:G_F\longrightarrow GL_n(A)$ be a
lift of $\rb.$ Then $\rho$ is strictly equivalent to an upper triangular
representation.\label{borel1}\end{proposition}

\begin{proof} Let $J$ be an ideal of $A$ killed by
$\mathfrak{m}_A,$ and assume that $\rho\,\textrm{mod}\,J$ is upper triangular. The obstruction to
lifting $\rho\,\textrm{mod}\,J$ to $B_n(A)$ is given by an element
\[
\xi\in H^2\left(G_F,b\rb\right)\otimes J.
\]
 The obstruction $\xi$ is
trivial because the image of $\xi$ in $H^2(G_F,\Ad\rb)\otimes J$ is trivial, and
$H^*(G_F,\Hom(k\e_j,k\e_i))=(0)$ if $j>i.$ Thus there is an upper triangular lift $\rho'$ of
$\rho\,\textrm{mod}\,J.$ Now $\rho=(I+\psi)\rho'$ with
\[
\psi\in
H^1(G_F,\Ad\rb)\otimes J=H^1\left(G_F,b\rb\right)\otimes J,
\]
 and the
proposition follows.
\end{proof}

Let $\B$ be the full subcategory of $\rep_n(G_F;k)$ with objects $(A,\rho)$ satisfying
$\rho\mod{\m_A}=\rb$ and $\rho$ Borel i.e.
\[
\rho\sim \begin{pmatrix}
\omega^{n-1}&*&&*\\
&\omega^{n-2}&& \\
&&\ddots&\vdots\\
&&& 1\end{pmatrix}.
\]
It is easy to see that $\B$  determines deformation conditions
for $\rb.$  We shall refer to the deformation condition $\B$ as the {\em Ramakrishna
condition}. (When $n=2,$ these are the deformation conditions discussed in section 3 of
\cite{ramakrishna}.)

\begin{proposition} $\B0$ is smooth and its tangent space is
$$\bigoplus_{i=1}^{n-1}H^1\left(G_F,\Hom(k\e_{i+1},k\e_i)\right).$$
\label{borel2}\end{proposition}

\begin{proof}
 Let
 $\rho:G_F\longrightarrow GL_n(B)$ be a representation, say
 $\rho=\left(b_{ij}\omega^{j-i}\right)$ where
 $b_{ij}:G_F\longrightarrow B$ are functions with
\[
 b_{ij}(\sigma)=\left\{\begin{matrix} 0,\ \textrm{if}\ i>j,\\
1,\ \textrm{if}\ i=j\end{matrix}\right.
\]
 for any $\sigma\in G_F.$ Each $b_{i,i+1}\in
H^1(G_F,B(1)).$ The calculation in example E4 of \cite{taylorartin} shows that for a surjection
$f:A\longrightarrow B,$ the map
\[
H^1(G_F, A(1))\longrightarrow H^1(G_F, B(1))
\]
 is surjective. If
we assume $f$ to be small, it follows that the obstruction to there being a lift of type $\B$
of $\rho$ to $A$ is given by an element of
\[
H^2\left(G_F,\oplus_{j-i\geq 2}\Hom(k\e_j,k\e_i)\right).
\]
 But
this cohomology group vanishes because
\[
\dim_kH^0(G_F,k(j-i)=\dim_k H^2(G_F, k(j-i))=\dim_kH^0(G_F,
k(i-j+1))=(0)
\]
 for $j-i\geq 2$ as $\bar\omega$ has order greater than $n.$ Consequently
\[
H^1\left(G_F,\oplus_{j-i\geq 1}\Hom(k\e_j,k\e_i)\right)=
\oplus_iH^1\left(G_F,\Hom(k\e_{i+1},k\e_i)\right),
\]
 from which the statement about the tangent
space follows.
\end{proof}

\section{Constructing global deformation conditions with trivial dual Selemer group}

In this section, $F$ is any number field, and $k$ is a finite field of characteristic $\ell.$ Let
$\rb:G_F\longrightarrow GL_N(k)$ be a representation and let
 $\chi:G_F\longrightarrow W^\times$ be a character such that
$\chi\pmod{\ell}=\det\rb.$ We follow the conventions used in section \ref{prelim global}.

We shall say that  a global deformation condition $\cD$ with determinant
$\chi$ for $\rb$ \emph{satisfies   the tangent space inequality} if the inequality
\begin{align}\label{tangent space inequality}
 \sum_{v\in\Sigma(\cD)}\dim T{\cD_v}\geq (N-2)+\sum_{v\in\Sigma(\cD)}
\dim H^0\left(G_v,
\Ad^0\rb\right)
\end{align}
 holds. Recall that $\Sigma(\cD)$ is the finite set consisting of those
primes $v$ of $F$ where $\cD_v$ is not unramified, primes  lying above $\ell$ and
$\infty$, and primes where $\rb$ and $\chi$ are ramified. By Wiles' formula \ref{wiles formula}, $\cD$ as
 satisfies the tangent space inequality if
\[
\dim H^1_{\{T\cD_v\}}\left(F,\Ad^0\rb\right)-
\dim H^1_{\{T\cD_v^\perp\}}\left(F,\Ad^0\rb(1)\right)\geq N-2.
\]

\begin{definition}\label{defn big}
The residual representation $\rb:G_F\longrightarrow GL_N(k)$ is said to be a \emph{big representation} if
the following properties hold:
\begin{enumerate}
\item[(R1)] $\Ad^0\rb$ is absolutely
irreducible and
\[
H^1\left(\Gal(F(\Ad^0\rb)/F),\Ad^0\rb\right)
=H^1\left(\Gal(F(\Ad^0\rb(1))/F),\Ad^0\rb(1)\right) =\left(0\right).
\]
\item[(R2)] There is a non-archimedean prime $w_0$ of $F$ with $w_0\nmid \ell$ such that
\[
\rb|_{w_0}\sim
\begin{pmatrix}
\bar\omega^{N-1}&&&\\
&\bar\omega^{N-2}&& \\
&&\ddots&\\
&&& 1\end{pmatrix}\otimes \overline\eta
\]
 where $\overline\eta$ is an unramified character, and
the mod $\ell$ cyclotomic character $\bar\omega$ has order strictly greater than $N.$
\end{enumerate}
\end{definition}

Note that if $\rb$ is big, then R2 implies that $F$ does not contain all $\ell$-th
roots of unity, that $\Ad^0\rb$ and $\Ad^0\rb(1)$ are inequivalent, and that
$\ell>N.$ Also if $\rb$ is big and $k'$ is a finite extension of $k,$ then the
extension of scalars of $\rb$ to $GL_N(k')$ is again a big representation. Further
examples of big representations are supplied by the following proposition:

\begin{proposition}\label{bigrepns}
\mbox{}
\begin{enumerate}
\item[(i)] Let $F$ be a number field, and fix an integer $N\geq 2.$
There is a constant $C$ such that if $k$ is a finite field of characteristic $\ell>C,$ then
any representation $\rb:G_F\longrightarrow GL_N(k)$ with $\textrm{Im}\rb$ containing
$SL_N(k)$ is a big representation.
\item[(ii)] Let $\rb: G_\Q\longrightarrow GL_3(k)$ be a representation with
$\textrm{Im}\rb$ containing
$SL_3(k).$  Assume that $\ell$, the characteristic of $k$, is at least $7$. Further,
assume that if $\ell=7$ then the fixed field of $\Ad^0\rb$ does not contain
$\cos(2\pi/7)$. Then $\rb$ is a big representation.
\end{enumerate}
\end{proposition}

\begin{proof} We fix some notation first:
\begin{align*}
\widetilde\rho&:=\textrm{the composite}\
\begin{CD}G_F@>{\rb}>>GL_N(k)@>>>PGL_N(k),\end{CD}\\
\widetilde\chi&:=\textrm{the determinant of}\ \widetilde\rho\ (\textrm{so}\
\widetilde\chi:G_F\longrightarrow k^\times/k^{\times
N}),\\
\lefteqn{F(\widetilde\chi, \bar\omega)\ (\textrm{resp.}\ F(\widetilde\chi),
F(\bar\omega), F(\widetilde\rho))}\\
&:=\textrm{the extension of}\ F\ \textrm{through which}\ \widetilde\chi\ \textrm{and}\
\bar\omega\ (\textrm{resp.}\
\widetilde\chi,\bar\omega, \widetilde\rho)\ \textrm{factors},\\
d&:=[F(\widetilde\chi):F],\\
(\ell-1)/e&:=[F(\widetilde\chi,\bar\omega):F(\widetilde\chi)].
\end{align*} We shall now
show that the proposition holds with $C=\max(5, 2edN+1).$

The extension $F(\widetilde\rho)/F(\widetilde\chi)$ has Galois group $PSL_N(k)$, and so
$F(\widetilde\rho),F(\widetilde\chi,\bar\omega)$ are linearly disjoint over
$F(\widetilde\chi).$ Since $\bar\omega\left(G_{F(\widetilde\chi)}\right)=\F_\ell^{\times
e},$ the image of the homomorphism
\[
\widetilde\rho\times \bar\omega: G_F\longrightarrow
PGL_N(k)\times \F_\ell^\times
\]
 contains $PSL_N(k)\times \F_\ell^{\times e}.$

Fix a generator $a$ of the cyclic group $\F_\ell^\times,$ and set $b=a^{2ed}.$ It is easy to see
that the projective image of the diagonal matrix
\[
\begin{pmatrix} b^{N-1}&&\\
&\ddots& &\\
&&b&\\
&&&1\end{pmatrix}
\]
 is an element of $PSL_N(k).$ By the Chebotarev density theorem, there is an
unramified prime $v$ such that
\[
\widetilde\rho(\Frob_v)=\begin{pmatrix} b^{N-1}&&\\
&\ddots& &\\
&&b&\\
&&&1\end{pmatrix}$$ and $\bar\omega(\Frob_v)=b.$ Hence
$$\rb|_{F_v}\sim\begin{pmatrix} \bar\omega^{N-1}&&\\
&\ddots& &\\
&&\bar\omega&\\
&&&1\end{pmatrix}\otimes \overline\eta
\]
 where $\overline\eta$ is an unramified character. Now
the order of $\bar\omega|_{F_v}$ is the order of $b,$ and this is greater than $N$ if
$2edN<\ell -1;$ so R2 holds. Since $\ell\geq 7,$ the vanishing of $H^1$ in  R1 holds by Theorem 4.2 of \cite{CPR}.

For the sepecial case when $N=3$ and $F=\Q$, note that $d=[\Q(\tilde\chi):\Q]$ is
either $1$ or $3$, and since $\ell\geq 7$ we must have
 $[\Q(\tilde\chi,\bar\omega):\Q(\tilde\chi)]\geq 4$ except in the case
  $\Q(\tilde\chi)=\Q(\cos(2\pi/7))$ (which we are excluding). Hence the image of
  $\widetilde\rho\times \bar\omega$ contains an element of the form
  \[
  \begin{pmatrix}
  a&0&0\\
  0&1&0\\
  0&0&a^{-1} \end{pmatrix}\times a
  \]
  where $a\in \F_\ell^\times$ has order at least $4$. The rest of the proof is then as before.

\end{proof}

\begin{remark} Keep the notation introduced in the proof of Proposition~\ref{bigrepns}. Since
$F(\widetilde\chi,\bar\omega)\supset\Q(\bar\omega),$ we have $nd\geq e$ and
$d\leq N$. Hence if
 $\ell>2[F:\Q]N^3+1$ then $2edN<\ell-1$  and so any
representation $\rb:G_F\longrightarrow GL_N(k)$ with $\textrm{Im}\rb $ containing
$SL_N(k)$ is a big representation.

\end{remark}

\begin{proposition}\label{smooth1}
Let $\rb:G_F\longrightarrow GL_N(k)$ be a big representation, and let
$\chi:G_F\longrightarrow W^\times$ be a character lifting $\det\rb.$ Fix a prime $w_0$ of $F$ such that  $\rb|_{w_0}$ satisfies condition R2  of Defiition \ref{defn big}.

If
$\cD_0$ is a global deformation condition with determinant $\chi$ for $\rb$
satisfying the tangent space inequality, then there  exists a global deformation condition
$\cD$ with determinant $\chi$ for $\rb$ with $\Sigma(\cD)\supseteq
\Sigma(\cD_0)$
 such that:
\begin{itemize}
\item  If $v\in\Sigma(\cD_0)$ then $\cD_{0v}=\cD_{v};$
 \item If $v\in\Sigma(\cD)-\Sigma(\cD_0)$ then $\cD_v$ is smooth and $\rb(\Frob_{v})=\rb
(\Frob_{w_0}).$ Furthermore,  the tangent space $T{\cD_v}$ satisfies
\[
H^1(F_v,\Ad^0\rb)=H^1_\textrm{nr}(F_v,\Ad^0\rb)\oplus T{\cD_v};
\]
\item We have $H^1_{\{T{\cD_v}^\perp\}}\left(F,\Ad^0\rb(1)\right)=(0)$.
\end{itemize}
\end{proposition}

\begin{proof} If we can find
\[
0\neq \xi\in
H^1_{\{T\cD_{0,v}^\perp\}}\left(F,\Ad^0\rb(1)\right),
\]
 then using Wiles'
formula \ref{wiles formula}, we see that
\[
\dim_kH^1_{\{T\cD_{0v}\}}\left(F,\Ad^0\rb\right)\geq
N-1.
\]
 Using Theorem~\ref{Gal coh thm}, we can find a prime $w_1\not\in\Sigma_0$  such that
\begin{itemize}
\item[(a)] $\rb(\Frob_{w_1})=\rb (\Frob_{w_0})$ and
$\bar\omega(\Frob_{w_1})= \bar\omega(\Frob_{w_0});$\item[(b)] The restriction
\[
H^1_{\{T\cD_{0v}\}}\left(F,\Ad^0\rb\right)\longrightarrow
H^1_\textrm{nr}\left(F_{w_1},\Ad^0\rb\right)
\]
 is surjective; and, \item[(c)] The image
of $\xi$ when restricted to $H^1_\textrm{nr}\left(F_{w_1},\Ad^0\rb(1)\right)$ is
non-trivial.
\end{itemize}

Let $\cD_1$ be the deformation condition for $\rb$ with determinant $\chi$ with
the following local conditions: At primes not  equal to $w_1$, the local deformation conditions
$\cD_{0v}$ and $\cD_{1v}$ are the same. At the prime $w_1,$ the local deformation
condition $\cD_{1w_1}$ is determined by a Ramakrishna condition (cf
subsection~\ref{ramakrishna}). Thus $\cD_1$ is smooth at $w_1.$ The proof now proceeds as
in Lemma~1.2 of \cite{taylorartin}: Denote by $\{\mathcal{S}_v\}$ the local Selmer conditions
\[
\mathcal{S}_v=
\begin{cases}
T\cD_{0v},& \text{if $v\neq w_1$;}\\
(0),& \text{if $v=w_1$.}
\end{cases}
\]
Using Wiles' formula \ref{wiles formula},
\begin{align*}
\lefteqn{\dim\,H^1_{\{\mathcal{S}_v\}}\left(F,\Ad^0\rb\right)
-\dim\,H^1_{\{\mathcal{S}_v^\perp\}}\left(F,\Ad^0\rb(1)\right)}\\
&=\sum_{v\nmid\infty}\!(\dim\,\mathcal{S}_v-\dim\,H^0(F_v,\Ad^0\rb))-\sum_{v|\infty}H^0\left(F_v,
\Ad^0\rb\right)\\
&=\dim\,H^1_{\{T\cD_{0v}\}}\left(F,\Ad^0\rb\right)
-\dim\,H^1_{\{T\cD_{0v}^\perp\}}\left(F,\Ad^0\rb(1)\right)
-\dim\,H^1_\textrm{nr}\left(F_{w_1},\Ad^0\rb\right),
\end{align*}
 and
by (b), the sequence
\[
0\longrightarrow H^1_{\{\mathcal{S}_v\}}\left(F,\Ad^0\rb\right)
\longrightarrow H^1_{\{\mathcal{L}_{0v}\}}\left(F,\Ad^0\rb\right) \longrightarrow
H^1_\textrm{nr}\left(F_{w_1},\Ad^0\rb\right)\longrightarrow 0
\]
 is exact. Hence we have
\[
H^1_{\{\mathcal{S}_v^\perp\}}\left(F,\Ad^0\rb(1)\right)=
H^1_{\{T\cD_{0v}^\perp\}}\left(F,\Ad^0\rb(1)\right).
\]
 Using condition (c) along
with
$H^1(F_{w_1},\Ad^0\rb(1))=H^1_\textrm{nr}\left(F_{w_1},\Ad^0\rb(1)\right)
\oplus T\cD_{1w_1}^\perp,$ we see that
\[
0\neq \xi\not\in H^1_{\{T\cD_{1,v}^\perp\}}\left(F,\Ad^0\rb(1)\right)
\subseteq H^1_{\{T\cD_{0,v}^\perp\}}\left(F,\Ad^0\rb(1)\right),
\]
 and the
proposition follows inductively.
\end{proof}

An application of Theorem~\ref{thmboeckle} then gives the following:
\begin{theorem} We keep the notations and assumptions of Proposition~\ref{smooth1}
 above.
If for each $v\in\Sigma(\cD_0)$ the local deformation condition $\cD_{0v}$ is
smooth, then the universal deformation ring for deformations of type $\cD$ is a power
series ring over $W$ in
\[
\sum_{v\in\Sigma(\cD_0)}\dim_kT\cD_{0v}
-\sum_{v\in\Sigma(\cD_0)}\dim_kH^0(F_v,\Ad^0\rb)
\]
variables.\label{smooth2}\end{theorem}

\section{Lifting Galois representations to characteristic 0}

\subsection{Proof of the main theorem} \label{proof of main thm}
 We now fix a number field $F,$ an integer $N\geq 3,$ and
 a finite field $k$   of characteristic $\ell\geq N^{3[F:\Q]N}.$
We first prove a variation of the main theorem: Suppose we are given  a  representation
 $ \rb:G_F\longrightarrow GL_N(k)$
 and a lifting
 $\chi:G_F\longrightarrow W^\times$
 of $\det\rb$ which is minimally ramified away from $\ell.$
 Assume that $\rb$ satisfies the following hypotheses:
\begin{enumerate}
\item[(H0)] For any open subgroup $H\leq G_F$ all irreducible components of the semi-simplification
of $\rb|_H$ are absolutely irreducible;
\item[(H1)] $\rb:G_F\longrightarrow
GL_N(k)$ is a big representation;
\item[(H2)] $\rb$ is not totally even; and
\item[(H3)]
For every prime $v|\ell,$ we have
\[
H^0\left(F_v,\Ad^0\rb(1)\right)=\left(0\right);
\]
\end{enumerate}
\begin{claim} \label{main thm claim}
Under the above assumptions,  there is a global deformation condition $\cD$
with determinant $\chi$ for $\rb$ such that the universal deformation ring is a power
series ring over $W$ in
\[
\sum_{v\in\Sigma(\cD)}\dim_k\cD_{v}-
\sum_{v\in\Sigma(\cD)}\dim_kH^0(F_v,\Ad^0\rb)\geq N-2
\]
variables.
\end{claim}

\medskip
The main theorem then follows immediately from the claim by extension of scalars,
Proposition \ref{bigrepns}, and Lemma \ref{scalar extn}.
\bigskip
\begin{proof}[Proof of claim \ref{main thm claim}] Observe that $\ell\geq N^{3[F:\Q]N}$ implies
$[F_v(\zeta_\ell):F_v]\geq 3N$ for every $v|N!.$ Now let $\cD_0$ be the deformation
condition with determinant $\chi$ for $\rb$ given by the following local conditions:
\begin{itemize}
 \item At a prime $v|\ell,$ the local deformation condition is given by
the single restriction that the determinant is $\chi.$  \item At a prime $v$ where
$\rb$ is ramified, the local condition $\cD_{0v}$ is the one given by
Theorem~\ref{main local} \item $\cD_{0v}$ is unramified at all other primes.
\end{itemize}

Let $v$ be a prime of $F$ lying above $\ell.$ By assumption H3 and local duality, we have
\[
\dim H^2\left(F_v,\Ad^0\rb\right)=\dim
H^0\left(F_v,\Ad^0\rb(1)\right)=0.
\]
 Hence the deformation condition $\cD_{0v}$
is smooth and, by the local Euler characteristic formula, we have
\[
\dim T\cD_{0v}-\dim
H^0\left(F_v,\Ad^0\rb\right)=[F_v:\Q_\ell](N^2-1).
\]
 Adding up over primes above $\ell$, we get
\[
\sum_{v|\ell}\dim T\cD_{0,v}-\sum_{v|\ell}\dim
H^0\left(F_v,\Ad^0\rb\right)= [F:\Q](N^2-1).
\]

We are assuming that $\rb$ is not totally even. We can therefore find  a real prime
$\infty_\R$ of $F,$ a choice $c\in G_F$ of complex conjugation under the embedding given by
$\infty_\R$ such that $\rb(c)$ is not a scalar. Let $m$ be the number of $+1$
eigenvalues of $\rb(c).$ Now
\begin{align*}
\lefteqn{\sum_{v|\infty}\dim H^0\left(F_v, \Ad^0\rb\right)}\\
&\leq ([F:\Q]-1)(N^2-1)+\dim
H^0\left(F_{\infty_\R},\Ad^0\rb\right)\\
&=([F:\Q]-1)(N^2-1)+ m^2+(N-m)^2-1.
\end{align*}

Finally, $\dim T\cD_{0v}=\dim H^0(F_v,\Ad^0\rb)$ for $v\in\Sigma(\cD_0),
v\nmid \ell\infty.$ Hence
\begin{align*}
\lefteqn{\sum_{v\in\Sigma(\cD_0)}\dim T\cD_{0v}-\sum_{v\in\Sigma(\cD_0)}\dim
H^0\left(F_v,\Ad^0\rb\right)}\\
&\geq [F:\Q](N^2-1)-([F:\Q]-1)(N^2-1)- m^2-(N-m)^2+1\\
&=2m(N-m).
\end{align*}
From $(m-1)(N-m-1)\geq 0,$ we get $m(N-m)\geq N-1,$ and consequently
$\cD_0$ satisfies the tangent space inequality.

Applying Cor~\ref{smooth2}, we obtain a  deformation condition $\cD$ with determinant
$\chi$ such that the universal deformation ring is a power series ring over $W$ in
\begin{align*}
\lefteqn{\sum_{v\in\Sigma(\cD)}\dim_k\cD_{v}-
\sum_{v\in\Sigma(\cD)}\dim_kH^0(F_v,\Ad^0\rb)}\\
&=\sum_{v\in\Sigma(\cD_0)}\dim_k\cD_{v}-
\sum_{v\in\Sigma(\cD_0)}\dim_kH^0(F_v,\Ad^0\rb)\\
&\geq N-2
\end{align*}
variables. \end{proof}

\subsection{A lifting result when $N=3$ and $F=\Q$}
 We now discuss how to improve on the main theorem  for the case when $N=3$ and
$F=\Q$. From here on, $k$ is a finite field of characteristic $\ell$. An odd representation is one with complex conjugation having two distinct eigenvalues.

\begin{theorem} \label{lifting GL3}
Let $\rb:G_\Q\longrightarrow GL_3(k)$ be an odd representation with
image of $\rb$ containing  $SL_3(k)$ and let $\chi:G_\Q\longrightarrow W^\times$ be a character lifting the determinant of $\rb$. Suppose that $\ell\geq 7$, and further assume that if $\ell=7$ then the fixed field of $\Ad^0\rb$ does not
contain $\cos(2\pi/7)$.
Then there is a continuous representation
$\rho:G_\Q\longrightarrow GL_3(W)$ with determinant $\chi$, unramified outside finitely
many primes, such that $\rho\pmod{\ell}=\rb$.

\end{theorem}

\begin{proof}
As in the proof of the main theorem, we may extend scalars and assume that $\rb:G_\Q\longrightarrow GL_3(k)$ satisfies the three conditions H0, H1 (by Proposition~\ref{bigrepns}) and H2 of the preceding section \ref{proof of main thm}. Thus $\rb$ is a big odd representation  such that for any open subgroup $H\leq G_\Q$ all irreducible components in the semi-simplification
of $\rb|_H$ are absolutely irreducible.
We will now find a global deformation condition
$\cD_0=\{\cD_{0p}\}$ with determinant $\chi$, smooth local conditions and satisfying the
tangent space inequality. There is no issue at a primes away from $2$, $3$ and $\ell$:
 If $p>3$ and
$p\neq \ell$  we take $\cD_{0p}$ to be the one obtained through
Theorem~\ref{main local} or the unramified deformation condition depending on the
ramification of $\rb$ at $p$.

Now let $p$ be one of $2$, $3$ or $\ell$. We write $\rb_p$ for the restriction of $\rb$ to $\Q_p$.
 If $H^2(\Q_p,\Ad^0\rb_p)=(0)$ then we take $\cD_{0p}$ to
be the class of liftings with determinant $\chi$ (cf. Example~\ref{ex 1}). Thus
$\cD_{0p}$ is smooth and
\[
\dim T\cD_{0p}=\delta_{p\ell}\dim \Ad^0\rb +\dim H^0(\Q_p,\Ad^o\rb),
\]
 and we can proceed as in the proof of the main theorem.

So lets assume now that $H^2(\Q_p,\Ad^0\rb)\neq (0)$. Thus
\[
\Hom_{G_{\Q_p}}(\rb,\rb(1))\cong H^0(\Q_p,\Ad^0\rb(1))\neq (0).
\]
Consideration of a non-trivial morphism from $\rb_p$ to $\rb_p(1)$ then implies that
we may assume, conjugating if necessary, one  of the following holds.
\begin{description}
\item[Type A] $\rb_p=\begin{pmatrix}
1&*&*\\
0&\bar\omega&*\\
0&0&\bar\omega^2
\end{pmatrix}\eta$.
\item[Type B] $\rb_p=\begin{pmatrix}
1&x&y\\
0&\varepsilon &z\\
0&0&\bar\omega
\end{pmatrix}\eta$ where $x$ is non-split if $\varepsilon=\bar\omega^{-1}$ and $z$ is
non-split if $\varepsilon=\bar\omega^2.$
\item[Type C] $\rb_p\sim\rb_p(1)$. Comparing determinants gives $\bar\omega^3=1$ and
so we have $(p,\ell)=(2,7)$ or $(3,13)$.  We may assume that $\rb_p$ is absolutely irreducible (the reducible case is covered already). In this case,
$\rb_p$ is induced from a character of $G_{\Q_p(\zeta_\ell)}.$
\end{description}

Suppose now that $p=2$ or $3$. If $\rb_p$ is of Type A or of Type B with $\varepsilon$
unramified then $\rb_p$ is a twist of a tamely ramified representation and we can use
Theorem~\ref{main local} to get a smooth deformation condition $\cD_{0p}$ with
determinant $\chi$ and $\dim T\cD_{0p}=\dim H^0(\Q_p,\Ad^0\rb)$.

If $\rb_p$ is of Type C or of Type B with $\varepsilon$ ramified, then  $\ell$ does not
divide the order of the image of inertia under $\rb_p$. (If $\rb_p$ is Type B with
$\varepsilon$ ramified then we can assume $x=z=0$
since $H^1(\Q_p,k(\psi^{-1}))$ and $H^1(\Q_p,k(\psi\bar\omega^{-1}))$ are both trivial,
and then we can make $y=0$ because $H^1(\Q_p,k(\bar\omega^{-1}))=(0)$.) The construction and argument then proceeds as in \cite[Example E1]{taylorartin}:
Take
$K$ to be fixed field of $\rb_p$ adjoined the maximal unramified extension of $\Q_p$,
and then take $\cD_{0p}$ to be lifts of $\rb_p$ which factor through
$\textrm{Gal}(K/\Q_p)$ and determinant $\chi$. This is a smooth deformation condition
and its tangent space has dimension $\dim H^0(\Q_p,\Ad^0\rb)$.

We now choose local conditions at $\ell$ and define a $G_{\Q_\ell}$ subspace $N$ of
$\Ad^0\rb$ as follows. (The same constructions work when $p=2$ or $3$ provided $\bar\omega^3\neq 1$.)
\begin{enumerate}
\item[(a)] Suppose $\rb_\ell$ is either of Type A or of Type B with $\varepsilon$ different from $1$ or $\bar\omega$ or $ \bar\omega^{-1}$ or $\bar\omega^2$. Take $\cD_{0\ell}$ to be upper triangular deformations of $\rb$ with determinant $\chi$ and set $N$ to be the space of trace $0$ upper triangular matrices in $\Ad^0\rb$.
    \item[(b)] Suppose $\rb_\ell$ is of Type B and $\varepsilon$ is $1$ or $\bar\omega^{-1}$. Take $\cD_{0\ell}$ to be deformations of the form $\begin{pmatrix} *&*&*\\ *&*&*\\ 0&0&*\end{pmatrix}$  with determinant $\chi$, and set $N$ to be the matrices of the same form in $\Ad^0\rb$.
        \item[(c)] Suppose $\rb_\ell$ is of Type B and $\varepsilon$ is $\bar\omega$ or $\bar\omega^2$. Take $\cD_{0\ell}$ to be deformations of the form $\begin{pmatrix} *&*&*\\ 0&*&*\\ 0&*&*\end{pmatrix}$  with determinant $\chi$, and set $N$ to be the matrices of the same form in $\Ad^0\rb$.
\end{enumerate}
We leave out the verification that these are in fact deformation conditions as defined in Section~\ref{gen DC}.
The tangent space $T\cD_{0\ell}\subseteq H^1(\Q_\ell,\Ad^0\rb)$ is the image of $H^1(\Q_\ell, N)$ in $H^1(\Q_\ell,\Ad^0\rb)$ under the inclusion $N\subseteq \Ad^0\rb$.

By considering the composition series for
$N$, we see that $N$ has no quotient isomorphic to $k(1)$. (When $\rb$ is of Type B and $\varepsilon$ is $\bar\omega^{-1}$ or $\bar\omega^2$ we need to use the non-splitting of $x$ and $z$.) Consequently $H^2(\Q_\ell, N)=(0)$ and $\cD_{0\ell}$ is a smooth deformation condition (cf \cite[Theorem 1.2]{boeckle2}). Also the composition series for $\Ad^0\rb/N$ shows that $H^0(\Q_\ell, \Ad^0\rb/N)=(0)$. Hence, from the exact sequence
\[
0\longrightarrow N\longrightarrow \Ad^0\rb\longrightarrow \Ad^0\rb/N\longrightarrow 0
\]
we see that $H^0(\Q_\ell, \Ad^0\rb)\cong H^0(\Q_\ell, N)$ and $T\cD_{0\ell}\cong H^1(\Q_\ell, N)$.
It then follows, by the local Euler characteristic formula, that
\begin{equation}\label{inequality}
 \dim T\cD_{0\ell}-\dim H^0(\Q_\ell, \Ad^0\rb)
=\dim N+ \dim H^2(\Q_\ell ,N)=\dim N\geq 5.
\end{equation}

The result follows from Theorem~\ref{smooth2}  once we verify that the deformation condition $\cD_0=\{\cD_{0p}\}$
satisfies the tangent space inequality~\eqref{tangent space inequality}. Away from $\ell$ and $\infty$ we have $\dim T\cD_{0p}=\dim H^0(\Q_p,\Ad^0\rb)$. So we need to verify that
\begin{equation*}
\dim T\cD_{0\ell}\geq 1 +\dim H^0(\Q_\ell,\Ad^0\rb) +\dim H^0(\R,\Ad^0\rb),
\end{equation*}
and that follows from \eqref{inequality} since $\dim H^0(\R,\Ad^0\rb)=4$ as $\rb$ is not totally even.
\end{proof}

\bibliographystyle{amsplain}
\bibliography{references}

\end{document}